\documentclass[smallextended,numbook]{svjour3}     
\smartqed  
\usepackage{graphicx}
\usepackage{amsfonts}
\usepackage{amssymb}
\usepackage{amsmath}
\usepackage{amsxtra}
\usepackage{subfigure}
\usepackage{epstopdf}
\usepackage{latexsym}
\begin{document}

\title{ On the iteratively regularized Gauss--Newton method in Banach spaces with
applications to parameter identification problems
}


\author{Qinian Jin   \and Min Zhong 
}


\institute{Qinian Jin \at
Mathematical Sciences Institute,
Australian National University, Canberra, ACT 0200, Australia\\
\email{Qinian.Jin@anu.edu.au}\\
\\
Min Zhong \at
School of Mathematical Sciences, Fudan University, Shanghai 200433, China\\
\email{09110180007@fudan.edu.cn}\\
}


\newtheorem{Assumption}{Assumption}[section]
\newtheorem{Rule}{Rule}[section]
\newtheorem{Test}{Test}[section]

\def\l{\langle}
\def\r{\rangle}
\def\a{\alpha}
\def\b{\beta}
\def\d{\delta}
\def\la{\lambda}
\def\p{\partial}
\def\vep{{\mathcal E}}
\def\R{{\mathcal R}}

\def\N{\mathcal N}
\def\R{\mathcal R}
\def\D{\mathcal D}
\def\X{\mathcal X}
\def\Y{\mathcal Y}
\def\B{\mathcal B}
\def\A{\mathcal A}

\maketitle

\begin{abstract}

In this paper we propose an extension of the iteratively regularized Gauss--Newton method
to the Banach space setting by defining the iterates via convex optimization problems. We
consider some a posteriori stopping rules to terminate the iteration and present the detailed
convergence analysis. The remarkable point is that in each convex optimization problem
we allow non-smooth penalty terms including $L^1$ and total variation (TV) like penalty functionals.
This enables us to reconstruct special features of solutions such as sparsity and discontinuities
in practical applications. Some numerical experiments on parameter identification in partial
differential equations are reported to test the performance of our method.

\end{abstract}

\def\theequation{\thesection.\arabic{equation}}
\catcode`@=12

\section{\bf Introduction}
\setcounter{equation}{0}

Inverse problems arise from many practical applications whenever one searches for unknown
causes based on observation of their effects. A characteristic property of inverse problems
is their ill-posedness in the sense that their solutions do not depend continuously on the data.
Due to errors in the measurements, in practical applications one never has the exact data;
instead only noisy data are available. Therefore, how to use the noisy data to produce a stable
approximate solution is an important topic.

We are interested in solving nonlinear inverse problems in Banach spaces which
can be formulated as the nonlinear operator equation
\begin{equation}\label{1.1}
F(x)=y,
\end{equation}
where $F: D(F)\subset \X\mapsto \Y$ is a nonlinear operator between two Banach spaces $\X$ and $\Y$
with domain $D(F)\subset \X$. We will use the same notation $\|\cdot \|$ to denote the norms of $\X$
and $\Y$ which should be clear from the context.  Let $y^\d$ be the only available
approximate data to $y$ satisfying
\begin{equation}\label{1.3}
\|y^\d-y\|\le \d
\end{equation}
with a given small noise level $\delta> 0$. Due to the ill-posedness, regularization methods should be
employed to produce from $y^\d$ a stable approximate solution.

When both $\X$ and $\Y$ are Hilbert spaces and $F$ is Fr\'{e}chet differentiable, a lot of regularization methods have been developed
during the last two decades, see \cite{BK04,Jin2011,Jin2012,JT09,KNS2008} and the references therein.
The iteratively regularized Gauss--Newton method is one of the well known methods and it takes
the form (\cite{B92})
$$
x_{n+1}^\d=x_n^\d- \left(\a_n  +F'(x_n^\d)^* F'(x_n^\d) \right)^{-1}
\left(F'(x_n^\d)^*(F(x_n^\d)-y^\d) +\a_n (x_n^\d-x_0)\right),
$$
where $F'(x)$ denotes the Fr\'{e}chet derivative of $F$ at $x$, $F'(x)^*$ denotes the adjoint
of $F'(x)$, $x_0^\d:=x_0$ is an initial guess, and $\{\a_n\}$ is a sequence of positive numbers satisfying
\begin{equation}\label{1.5}
\a_n>0, \qquad 1\le \frac{\a_n}{\a_{n+1}}\le \theta \qquad \mbox{and}\qquad \lim_{n\rightarrow\infty} \a_n=0
\end{equation}
for some constant $\theta>1$. When terminated by the discrepancy principle, the regularization property of
the iteratively regularized Gauss--Newton method has been studied extensively, see \cite{JT09,KNS2008} and
references therein. It is worthwhile to point out that $x_{n+1}^\d$ is the unique minimizer of the quadratic functional
\begin{equation}\label{4.1.1}
\|y^\d-F(x_n^\d) -F'(x_n^\d) (x-x_n^\d)\|^2 +\a_n \|x-x_0\|^2 \qquad \mbox{over } \X.
\end{equation}

Regularization methods in Hilbert spaces can produce good results
when the sought solution is smooth. However, because such methods have a tendency to over-smooth solutions,
they may not produce good results in applications where the sought solution has special features such as sparsity
or discontinuities. In order to capture the special features, the methods in Hilbert spaces must be modified
by incorporating the information of some adapted penalty functionals such as the $L^1$ and the total
variation (TV) like functionals, for which the theories in Hilbert space setting are no longer applicable.
On the other hand, due to their intrinsic features, many inverse problems are more natural to formulate in Banach
spaces than in Hilbert spaces. Therefore, it is necessary to develop regularization methods to solve inverse
problems in the framework of Banach spaces with general penalty function.

In this paper we will extend the iteratively regularized Gauss--Newton method to the Banach space setting.
Motivated by the variational formulation (\ref{4.1.1}) in Hilbert spaces, it is natural to use convex optimization
problems to define the iterates. To this end, we take a proper, lower semi-continuous, convex function
$\Theta: \X\to (-\infty, \infty]$ whose sub-differential is denoted as $\p \Theta$. By picking an initial guess
$x_0\in D(F)\cap D(\Theta)$ and $\xi_0\in \p \Theta(x_0)$, we define
\begin{equation}\label{4.5.3}
x_{n+1}^\d:=\arg \min_{x\in \X} \left\{ \|y^\d-F(x_n^\d)-F'(x_n^\d) (x-x_n^\d)\|^p
+\a_n D_{\xi_0} \Theta (x, x_0)\right\}
\end{equation}
where $1\le p<\infty$, $x_0^\d=x_0$, and $D_{\xi_0} \Theta(x,x_0)$ denotes the Bregman distance induced by $\Theta$
at $x_0$ in the direction $\xi_0$. When $\Theta(x)=\|x-x_0\|^p$ and $\xi_0=0$, this method has been considered
in \cite{KH2010} under essentially the nonlinearity condition
\begin{equation}\label{4.5.2}
\|(F'(x)-F'(z)) h\| \le \kappa \|F'(z) (x-z)\|^{1/2} \|F'(z) h\|^{1/2}
\end{equation}
with the iteration terminated by an a priori stopping rule. It turns out that (\ref{4.5.2}) is difficult
to verify for nonlinear inverse problems, and the restriction of $\Theta$ to the special choice
may prevent the method from capturing the special features of solutions. Moreover, since a priori
stopping rules depend crucially on the unknown source conditions, it is useless in practical applications.
In this paper we will develop a convergence theory on the iteratively regularized Gauss--Newton method
in Banach spaces with general convex penalty function $\Theta$. We will propose some a posteriori stopping rules,
including the discrepancy principle, to terminate the method and give detailed convergence analysis
under reasonable nonlinearity conditions.

This paper is organized as follows. In section \ref{Sect2} we give some preliminary facts on convex analysis.
In section \ref{Sect3} we then formulate the iteratively regularized Gauss--Newton method in Banach spaces and
propose some a posteriori stopping rules. We show that the method is well-defined and obtain a weak convergence
result. In section \ref{Sect4} we derive the rates of convergence when the solution satisfies certain
source conditions formulated as variational inequalities. In section \ref{Sect5} we prove a strong convergence
result without assuming any source conditions when $\Y$ is a Hilbert spaces and $\Theta$ is a $2$-convex function,
which is useful for sparsity reconstruction and discontinuity detection. Finally, in section \ref{Sect6} we
present some numerical experiments to test our method for parameter identification in partial differential equations.

\section{\bf Preliminaries}\label{Sect2}
\setcounter{equation}{0}

Let $\X$ be a Banach space with norm $\|\cdot\|$. We use $\X^*$ to denote its dual space. Given $x\in \X$
and $\xi \in \X^*$ we write $\l \xi, x\r=\xi (x)$ for the duality pairing. If $\Y$ is another Banach space
and $A: \X\to \Y$ is a bounded linear operator, we use $A^*: \Y^*\to \X^*$ to denote its adjoint, i.e.
$\l A^* \zeta, x\r=\l \zeta, A x\r$ for any $x\in \X$ and $\zeta\in \Y^*$.

Let $\Theta: \X \to (-\infty, \infty]$ be a convex function. We use $D(\Theta):=\{x\in \X: \Theta(x)<+\infty\}$
to denote its effective domain. We call $\Theta$ proper if $D(\Theta)\ne \emptyset$. Given $x\in \X$ we define
$$
\p \Theta(x):=\{\xi\in \X^*: \Theta(z)-\Theta(x)-\l \xi, z-x\r \ge 0 \mbox{ for all } z\in \X\}
$$
which is called the subgradient of $\Theta$ at $x$. It is clear that $\p \Theta(x)$ is convex and closed in $\X^*$ for
each $x\in \X$. The multi-valued mapping $\p \Theta: \X\to 2^{\X^*}$ is called the subdifferential of $\Theta$.
It could happen that $\p \Theta(x)=\emptyset$ for some $x\in D(\Theta)$. We set
$$
D(\p\Theta):=\{x\in D(\Theta): \p \Theta(x)\ne \emptyset\}.
$$
For $x \in D(\p \Theta)$ and $\xi\in \p \Theta(x)$
we define
$$
D_\xi \Theta(z,x):=\Theta(z)-\Theta(x)-\l \xi, z-x\r, \qquad \forall z\in \X
$$
which is called the Bregman distance induced by $\Theta$ at $x$ in the direction $\xi$. Clearly $D_\xi \Theta(z,x)\ge 0$.
By direct calculation we can see that
\begin{equation}\label{4.3.1}
D_\xi \Theta(x_2,x)-D_\xi \Theta(x_1, x) =D_{\xi_1} \Theta(x_2,x_1) +\l \xi_1-\xi, x_2-x_1\r
\end{equation}
for all $x, x_1, x_2\in D(\p \Theta)$, $\xi\in \p \Theta(x)$, and $\xi_1\in \p \Theta(x_1)$.

A proper function $\Theta: \X\to (-\infty, \infty]$ is said to be $p$-convex for some $p\ge 2$
if there is a constant $C_0>0$ such that for all $x, z\in \X$ and $\la\in (0,1)$ there holds
$$
\Theta(\la z+(1-\la)x) +C_0 \la (1-\la) \|z-x\|^p \le \la \Theta(z) +(1-\la) \Theta(x).
$$
It can be shown that $\Theta$ is $p$-convex if and only if there is a constant $\gamma>0$ such that
\begin{equation}\label{pconv}
\|z-x\| \le \gamma \left[D_\xi \Theta(z,x)\right]^{\frac{1}{p}}
\end{equation}
for all $z\in \X$, $x\in D(\p \Theta)$ and $\xi\in \p \Theta(x)$.

For a proper, lower semi-continuous, convex function $\Theta: \X \to (-\infty, \infty]$ we can
define its Fenchel conjugate
$$
\Theta^*(\xi):=\sup_{x\in \X} \left\{\l\xi, x\r -\Theta(x)\right\}, \quad \xi\in \X^*.
$$
It is well known that $\Theta^*$ is also proper, lower semi-continuous, and convex. If, in addition,
$\X$ is reflexive, then $\xi\in \p \Theta(x)$ if and only if $x\in \p \Theta^*(\xi)$. When $\Theta$ is $p$-convex
satisfying (\ref{pconv}) with $p\ge 2$, it follows from \cite[Corollary 3.5.11]{Z2002} that $D(\Theta^*)=\X^*$, $\Theta^*$ is Fr\'{e}chet
differentiable and its gradient $\nabla \Theta^*: \X^*\to \X$ satisfies
\begin{equation}\label{3.29.1}
\|\nabla \Theta^*(\xi)-\nabla \Theta^*(\eta) \|\le \gamma^{\frac{p}{p-1}} \|\xi-\eta\|^{\frac{1}{p-1}},
\quad \forall \xi, \eta\in \X^*.
\end{equation}

Many examples of $p$-convex  functions can be provided by functions of the norms in $p$-convex Banach spaces.
We say a Banach space $\X$ is $p$-convex with $p\ge 2$ if there is a positive constant $c_p$ such that
$\d_{\X}(\varepsilon)\ge c_p \varepsilon^p$ for all $0\le \varepsilon \le 2$, where
$$
\d_{\X} (\varepsilon): =\inf \left\{ 2-\|x+z\|: x, z\in \X, \, \|x\|=\|z\|=1 \mbox{ and } \|x-z\|\ge \varepsilon\right\}
$$
is the modulus of convexity of $\X$. According to a characterization of uniform convexity of Banach spaces in
\cite{XR91}, it is easy to see that, for any $x_0\in \X$, the functional
$$
\Theta(x):=\|x-x_0\|^p
$$
is $p$-convex and its subgradient at $x$ is given by $\p \Theta(x) =p J_p(x-x_0)$, where $J_p: \X \to 2^{\X^*}$
denotes the duality mapping of $\X$ with gauge function $t\to t^{p-1}$ which is defined for each $x\in \X$ by
$$
J_p(x):=\left\{\xi\in \X^*: \|\xi\|=\|x\|^{p-1} \mbox{ and } \l \xi, x\r=\|x\|^p\right\}.
$$
The sequence spaces $l^q$, the Lebesgue spaces $L^q$, the Sobolev spaces $W^{k,q}$ and the Besov spaces $B^{s,q}$
with $1<q<\infty$ are the most commonly used function spaces that are $\max\{q,2\}$-convex (\cite{Adams,C1990}).

Given a proper, lower semi-continuous, $p$-convex function $\Theta$ on $\X$, we can produce such new functions
$\Theta':=\Theta +\Psi$ by adding any available proper, lower semi-continuous, convex functions $\Psi$ to $\Theta$.
In this way, we can construct non-smooth $p$-convex functions that can be used to detect special features
of solutions when solving inverse problems. For instance, let $\X=L^2(\Omega)$, where $\Omega\subset {\mathbb R}^N$
is a bounded domain in ${\mathbb R}^N$. It is clear that the functional
$$
x\to \int_\Omega |x(\omega)|^2 d\omega
$$
is $2$-convex on $L^2(\Omega)$. By adding the function $\int_\Omega |x(\omega)| d\omega$ to the multiple of the above function we can
obtain the $2$-convex function
$$
\Theta_1(x):=\lambda \int_\Omega |x(\omega)|^2 d\omega +\int_\Omega |x(\omega)| d\omega
$$
with small $\lambda>0$ which is useful for sparsity recovery (\cite{T96}). Similarly, we may produce on $L^2(\Omega)$
the $2$-convex function
$$
\Theta_2(x):= \lambda \int_\Omega |x(\omega)|^2 d\omega +\int_\Omega |D x|,
$$
where $\int_\Omega|D x|$ denotes the total variation of $x$ over $\Omega$ that is defined by (\cite{G84})
$$
\int_\Omega |D x| :=\sup\left\{ \int_\Omega x \mbox{div} \varphi d\omega: \varphi\in C_0^1(\Omega; {\mathbb R}^N) \mbox{ and }
\|\varphi\|_{L^\infty(\Omega)}\le 1\right\}.
$$
This functional is useful for detecting the discontinuities, in particular, when the solutions are piecewise-constant (\cite{ROF92}).

\section{\bf The method and its weak convergence}\label{Sect3}
\setcounter{equation}{0}

In this section we formulate the iteratively regularized Gauss--Newton method in the framework of Banach spaces
to produce a stable approximate solution of (\ref{1.1}) from an available noisy data $y^\d$ satisfying
(\ref{1.3}). In order to capture the features of solutions, we take a proper, lower semi-continuous, $p$-convex function
$\Theta: \X\to (-\infty, \infty]$ with $p\ge 2$; we assume that $\Theta$ satisfies (\ref{pconv})
and $D(F)\cap D(\p \Theta) \ne \emptyset$. We will work under the following conditions on the nonlinear operator $F$.

\begin{Assumption} \label{A1}

\begin{enumerate}
\item[]
\begin{enumerate}

\item[(a)] $D(F)$ is a closed convex set in $\X$ and the equation (\ref{1.1}) has a solution $x^\dag\in D(F)\cap D(\p \Theta)$;

\item[(b)] There is $\rho>0$ such that for each $x\in B_\rho(x^\dag)\cap D(F)$ there is a bounded linear operator
$F'(x): \X\to \Y$ such that
$$
\lim_{t\searrow 0} \frac{F(x+t(z-x))-F(x)}{t} =F'(x) (z-x), \qquad \forall z \in B_\rho(x^\dag) \cap D(F),
$$
where $B_\rho(x^\dag):=\{x\in \X: \|x-x^\dag\|< \rho\}$;

\item[(c)] The operator $T:=F'(x^\dag)$ is properly scaled so that $\|T\|\le \a_0^{1/p}/\gamma$;

\item[(d)] There exist two constants $K_0$ and $K_1$ such that
$$
\|[F'(z)-F'(x)] w\|\le K_0\|z-x\| \|F'(x) w\| +K_1\|F'(x)(z-x)\| \|w\|
$$
for all $w\in \X$ and $x, z\in B_\rho(x^\dag)\cap D(F)$.

\end{enumerate}
\end{enumerate}
\end{Assumption}

It is easy to see that condition (b) in Assumption \ref{A1} implies, for any $x, z\in B_\rho(x^\dag)\cap D(F)$, that
the function  $t\in (0,1) \to  F(x+t(z-x))\in \Y$ is differentiable and
$$
\frac{d}{d t} F(x+t(z-x)) =F'(x+t(z-x)) (z-x).
$$
The condition (d) was first formulated in \cite{JT09}.  In section 6 we will present several examples from the parameter
identification in partial differential equations to indicate that this condition indeed can be verified for a wide range
of applications. As direct consequences of (b) and (d), we have for $x, z\in B_\rho(x^\dag)\cap D(F)$ that
$$
\|F(z)-F(x)-F'(x) (z-x)\| \le \frac{1}{2} (K_0+K_1) \|z-x\| \|F'(x) (z-x)\|
$$
and
$$
\|F(z)-F(x)-F'(z) (z-x)\| \le \frac{3}{2} (K_0+K_1) \|z-x\| \|F'(x) (z-x)\|.
$$

In order to formulate the method, let
$$
\chi_{D(F)}(x)=\left\{\begin{array}{lll}
0, & x\in D(F),\\
+\infty, \quad & x\not\in D(F)
\end{array}\right.
$$
be the characteristic function of $D(F)$ and define
\begin{equation}\label{theta_F}
\Theta_F(x):=\Theta(x) +\chi_{D(F)}(x).
\end{equation}
Since $D(F)$ is closed and convex, $\chi_{D(F)}$ is a proper, lower semi-continuous, convex function on $\X$.
Consequently, $\Theta_F$ is a proper, lower semi-continuous, $p$-convex function on $\X$ satisfying
\begin{equation}\label{9.16.1}
\|z-x\|\le \gamma \left[D_\xi \Theta_F (z,x)\right]^{\frac{1}{p}}, \quad \forall z\in \X, x\in D(\p \Theta_F) \mbox{ and } \xi\in \p \Theta_F(x).
\end{equation}
We pick $\xi_0\in \X^*$ and define $x_0:=\nabla \Theta_F^*(\xi_0)$, where
$\Theta_F^*$ denotes the Fenchel conjugate of $\Theta_F$ and is known to be Fr\'{e}chet differentiable with
gradient $\nabla \Theta_F^*: \X^*\to \X$.  We have $x_0\in D(\p\Theta_F):=D(F)\cap D(\p\Theta)$ and $\xi_0\in \p \Theta_F(x_0)$.
Consequently
$$
x_0=\arg\min_{x\in \X} \left\{\Theta_F(x)-\l \xi_0, x\r\right\} =\arg \min_{x\in D(F)} \left\{\Theta(x)-\l \xi_0, x\r\right\}.
$$
We use $\xi_0$ and $x_0$ as initial data. We then pick a sequence of positive numbers $\{\a_n\}$ satisfying (\ref{1.5})
and define $\{x_n^\d\}$ successively by setting $x_0^\d:=x_0$ and letting $x_{n+1}^\d$ be the unique minimizer of the
convex minimization problem
\begin{equation}\label{IRGN}
\min_{x\in \X} \left\{\|y^\d-F(x_n^\d)-F'(x_n^\d) (x-x_n^\d)\|^p +\a_n D_{\xi_0} \Theta_F(x, x_0)\right\}.
\end{equation}
By the properties of $\Theta_F$, $x_{n+1}^\d$ is uniquely defined and $x_{n+1}^\d\in D(F)$.

Considering the practical applications, the iteration must be terminated by some a posteriori stopping rule
to output an integer $n_\d$ and hence $x_{n_\d}^\d$ which is used as an approximate solution of (\ref{1.1}).
In this paper we will consider the following three stopping rules.

\begin{Rule}\label{Rule1}
Let $\tau>1$ be a given number. We define $n_\d$ to be the integer such that
$$
\|F(x_{n_\d}^\d)-y^\d\| \le \tau \d <\|F(x_n^\d)-y^\d\|, \quad 0\le n<n_\d.
$$
\end{Rule}

\begin{Rule}\label{Rule2}
Let $\tau>1$ be a given number. If $\|F(x_0)-y^\d\|\le \tau \d$ we define $n_\d=0$;
otherwise we define $n_\d \ge 1$ to be the first integer such that
$$
\frac{1}{2} \left(\|F(x_{n_\d}^\d)-y^\d\| +\|F(x_{n_\d-1}^\d)-y^\d\|\right) \le \tau \d.
$$
\end{Rule}

\begin{Rule}\label{Rule3}
Let $\tau>1$ be a given number. If $\|F(x_0)-y^\d\|\le \tau \d$ we define $n_\d=0$;
otherwise we define $n_\d \ge 2$ to be the first integer such that
\begin{align}\label{MDP}
\max  \left\{\|F(x_{n_\d}^\d)-y^\d\|, \|F(x_{n_\d-1}^\d)-y^\d\| \right\} \le \tau \d.
\end{align}
\end{Rule}

Rule \ref{Rule1} is known as the discrepancy principle and is widely used to terminate
regularization methods. Rule \ref{Rule3} appeared first in \cite{K99} to deal with
some Newton-type regularization methods in Hilbert spaces. It is easy to see that
Rule \ref{Rule1} terminates the iteration no later than Rule \ref{Rule2}, and Rule \ref{Rule2}
terminates the iteration no later than Rule \ref{Rule3}. Most of the results in this paper are
true for Rule \ref{Rule1} except the ones in Section \ref{Sect4} concerning the rates of convergence
under certain source conditions formulated as variational inequalities; the convergence rates,
however, can be derived when the iteration is terminated by either Rule \ref{Rule2} or Rule \ref{Rule3}.

In this section we show that the method together with any one of the above three stopping rules with
$\tau>1$ is well-defined. To this end, we introduce the integer $\hat{n}_\d$ defined by
\begin{align}\label{nhat}
\a_{\hat{n}_\d}\le \frac{\mu^{-p} \d^p}{ \|\xi_0-\xi^\dag\|^{p^*}} <\a_n, \quad 0\le n<\hat{n}_\d,
\end{align}
where $p^*$ is the number conjugate to $p$, i.e. $1/p+1/p^*=1$, the number $\mu>0$ is chosen to satisfy
\begin{equation}\label{gamma}
\gamma^{\frac{1}{p-1}} \theta^{\frac{2}{p}} \mu^{-1} < \frac{\tau-1}{2}.
\end{equation}
and $\xi^\dag\in \p \Theta_F(x^\dag)$ is the unique element that realizes the distance $d(\xi_0, \p\Theta_F(x^\dag))$
from $\xi_0$ to the closed convex set $\p\Theta_F(x^\dag)$ in $\X^*$, i.e.
$$
d(\xi_0, \p\Theta_F(x^\dag)) =\|\xi_0-\xi^\dag\|.
$$
Because the sequence $\{\a_n\}$ satisfies (\ref{1.5}), the integer $\hat{n}_\d$ exists and is finite.
We will show that $x_n^\d \in B_\rho(x^\dag)\cap D(F)$ for all $0\le n\le \hat{n}_\d$ and $n_\d\le \hat{n}_\d$
for the integer $n_\d$ defined by any one of the above three stopping rules. For simplicity of presentation,
we use the notation $e_n^\d:=x_n^\d-x^\dag$. We also use $C$ to denote a universal constant that is
independent of $n$ and $\d$ when its explicit formula is not important.

\begin{lemma}\label{L3.1.1}
Let $\X$ and $\Y$ be Banach spaces, let $\Theta:\X\to (-\infty, \infty]$ be a proper,
lower semi-continuous, $p$-convex function with $p\ge 2$, let $\{\a_n\}$ be a
sequence satisfying (\ref{1.5}), and let $F$ satisfy Assumption \ref{A1}.
If $(p^* \gamma\mu + 2 \gamma^{\frac{p}{p-1}})  \|\xi_0-\xi^\dag\|^{\frac{1}{p-1}} <\rho$ and
$\vep:=(K_0+K_1) \|\xi_0-\xi^\dag\|^{\frac{1}{p-1}}$ is sufficiently small, then $x_n^\d\in
B_\rho(x^\dag)\cap D(F)$ and
\begin{align}
\|x_n^\d- x^\dag\| &\le \left(p^* \gamma\mu + 2 \gamma^{\frac{p}{p-1}}\right)  \|\xi_0-\xi^\dag\|^{\frac{1}{p-1}}, \label{est1}\\
\|T (x_n^\d-x^\dag)\|  & \le  \left(3 \mu +\gamma^{\frac{1}{p-1}} \right) \theta^{\frac{1}{p}}
\|\xi_0-\xi^\dag\|^{\frac{1}{p-1}} \a_n^{\frac{1}{p}}  \label{est2}
\end{align}
for all $0\le n\le \hat{n}_\d$. Moreover, $n_\d\le \hat{n}_\d$ for the integer $n_\d$ defined by either Rule \ref{Rule1},
\ref{Rule2}, or \ref{Rule3} with $\tau>1$.
\end{lemma}

\begin{proof}
Since $\xi^\dag\in \p \Theta_F(x^\dag)$ implies $x^\dag=\nabla \Theta_F^*(\xi^\dag)$, from the definition of
$x_0$ and (\ref{3.29.1}) it follows that
\begin{equation*}
\|x_0-x^\dag\| \le \gamma^{\frac{p}{p-1}} \|\xi_0-\xi^\dag\|^{\frac{1}{p-1}} < \rho.
\end{equation*}
Thus $x_0\in B_\rho(x^\dag)\cap D(F)$ and (\ref{est1}) holds. In view of the scaling condition
$\|T\|\le \a_0^{\frac{1}{p}}/\gamma$ we can obtain
\begin{equation}\label{9.15.2}
\|T e_0\| \le \gamma^{\frac{1}{p-1}}\a_0^{\frac{1}{p}}\|\xi_0-\xi^\dag\|^{\frac{1}{p-1}}.
\end{equation}
Therefore the result holds for $n=0$.
Now  we assume that the estimates for $x_n^\d$ have been proved for some $n<\hat{n}_\d$ and show that the estimates
for $x_{n+1}^\d$ are also true. By the minimizing property of $x_{n+1}^\d$ we have
\begin{align*}
\|y^\d-F&(x_n^\d)-F'(x_n^\d) (x_{n+1}^\d-x_n^\d)\|^p +\a_n D_{\xi_0} \Theta_F(x_{n+1}^\d, x_0)\\
& \le  \|y^\d-F(x_n^\d)-F'(x_n^\d) (x^\dag-x_n^\d)\|^p +\a_n D_{\xi_0} \Theta_F(x^\dag, x_0).
\end{align*}
By using the identity (\ref{4.3.1}) we have
$$
D_{\xi_0} \Theta_F(x_{n+1}^\d, x_0)-D_{\xi_0} \Theta_F(x^\dag, x_0) =D_{\xi^\dag} \Theta_F(x_{n+1}^\d, x^\dag)
-\l\xi_0-\xi^\dag, x_{n+1}^\d-x^\dag\r.
$$
Therefore, it follows from the above inequality that
\begin{align}\label{2.20.1}
 \|y^\d & - F(x_n^\d)-F'(x_n^\d) (x_{n+1}^\d-x_n^\d)\|^p +\a_n D_{\xi^\dag} \Theta_F(x_{n+1}^\d, x^\dag) \nonumber\\
& \le \|y^\d-F(x_n^\d)-F'(x_n^\d) (x^\dag-x_n^\d)\|^p +\a_n \l \xi_0-\xi^\dag, x_{n+1}^\d-x^\dag\r.
\end{align}
In view of the Young's inequality $ab\le \frac{1}{s} a^s +\frac{1}{t} b^t$ for $a, b\ge 0$, $s>1$ and
$\frac{1}{s}+\frac{1}{t}=1$, we have
\begin{align*}
\l \xi_0-\xi^\dag, x_{n+1}^\d-x^\dag\r & \le \frac{1}{p} \left(\gamma^{-1} \|e_{n+1}^\d\|\right)^p
+\frac{1}{p^*} \left(\gamma \|\xi_0-\xi^\dag\|\right)^{p^*}.
\end{align*}
Combining this with (\ref{2.20.1}) and using the $p$-convexity of $\Theta_F$, we can obtain
\begin{align*}
\|y^\d& - F(x_n^\d)-F'(x_n^\d) (x_{n+1}^\d-x_n^\d)\|^p
 + \frac{1}{p^*} \a_n \left(\gamma^{-1} \|e_{n+1}^\d\|\right)^p  \nonumber\\
& \le \|y^\d-F(x_n^\d)-F'(x_n^\d) (x^\dag-x_n^\d)\|^p
+ \frac{1}{p^*} \a_n \left(\gamma \| \xi_0-\xi^\dag\|\right)^{p^*}.
\end{align*}
By using the fact that $(a+b)^t\le a^t +b^t$ for $a, b\ge 0$ and $0\le t\le 1$, we have from
the above inequality that
\begin{align}\label{2.20.2}
\|e_{n+1}^\d\| \le \gamma \left(\frac{p^*}{ \a_n}\right)^{\frac{1}{p}}  \|y^\d-F(x_n^\d)-F'(x_n^\d) (x^\dag-x_n^\d)\|
 +\left(\gamma^p \|\xi_0-\xi^\dag\|\right)^{\frac{1}{p-1}}
\end{align}
and
\begin{align}\label{2.20.3}
\|  y^\d -F(x_n^\d)-F'(x_n^\d) (x_{n+1}^\d-x_n^\d)\|
&\le \|y^\d-F(x_n^\d)-F'(x_n^\d) (x^\dag-x_n^\d)\| \nonumber\\
&\quad \, +\left(\frac{1}{p^*} \gamma^{p^*} \a_n \|\xi_0-\xi^\dag\|^{p^*} \right)^{\frac{1}{p}}.
\end{align}
By using $\|y^\d-y\|\le \d$ and Assumption \ref{A1} we have
\begin{equation}\label{2.20.4}
\|y^\d-F(x_n^\d)-F'(x_n^\d) (x^\dag-x_n^\d)\|\le \d+ \frac{3}{2}(K_0+K_1) \|e_n^\d\| \|T e_n^\d\|.
\end{equation}
Since $n<\hat{n}_\d$, it follows from (\ref{nhat}) that
\begin{equation}\label{2.21.1}
\d\le \mu \|\xi_0-\xi^\dag\|^{\frac{1}{p-1}} \a_n^{\frac{1}{p}}.
\end{equation}
In view of the induction hypotheses we thus have
\begin{align*}
\|& y^\d-F(x_n^\d)-F'(x_n^\d) (x^\dag-x_n^\d)\|
\le \left(\mu + C \vep \right) \|\xi_0-\xi^\dag\|^{\frac{1}{p-1}} \a_n^{\frac{1}{p}}.
\end{align*}
Combining this with (\ref{2.20.2}) gives
\begin{align*}
\|e_{n+1}^\d\| & \le \left(\left(p^*\right)^{\frac{1}{p}} \gamma \mu + \gamma^{\frac{p}{p-1}} +C \vep \right)
\|\xi_0-\xi^\dag\|^{\frac{1}{p-1}}.
\end{align*}
Therefore, if $\vep$ is sufficiently small, then
$$
\|e_{n+1}^\d\| \le \left(p^* \gamma\mu + 2 \gamma^{\frac{p}{p-1}} \right)  \|\xi_0-\xi^\dag\|^{\frac{1}{p-1}} < \rho.
$$
Next we estimate $\|T e_{n+1}^\d\|$. From (\ref{2.20.3}) and (\ref{2.20.4}) it follows that
\begin{align}\label{2.19.2}
\| y^\d -  F(x_n^\d)-F'(x_n^\d) (x_{n+1}^\d-x_n^\d)\|
&  \le \d+ \frac{3}{2} (K_0+K_1) \|e_n^\d\| \|T e_n^\d\| \nonumber\\
& \quad \, +\left(\frac{1}{p^*} \gamma^{p^*} \a_n \|\xi_0-\xi^\dag\|^{p^*} \right)^{\frac{1}{p}}.
\end{align}
Observing that
\begin{align}\label{3.8.20}
\|y^\d-y -T e_{n+1}^\d \| & \le \|y^\d-F(x_n^\d)- F'(x_n^\d) (x_{n+1}^\d-x_n^\d)\| \nonumber \\
& \quad \, + \|y-F(x_n^\d)- F'(x_n^\d) (x^\dag-x_n^\d)\| \nonumber \\
& \quad \, + \|(T-F'(x_n^\d) ) e_{n+1}^\d\|.
\end{align}
Thus, we may use Assumption \ref{A1}, (\ref{2.19.2}), and the estimates on $\|e_n^\d\|$ and $\|e_{n+1}^\d\|$
to derive that
\begin{align}\label{2.20.5}
\|y^\d-y-T e_{n+1}^\d\| & \le \d+ C \vep \|T e_n^\d\| + C \vep \|T e_{n+1}^\d\|
 +\left(\frac{1}{p^*} \gamma^{p^*} \a_n \|\xi_0-\xi^\dag\|^{p^*} \right)^{\frac{1}{p}}.
\end{align}
Therefore, by using the induction hypothesis on $\|T e_n^\d\|$, the fact $\a_n\le \theta\a_{n+1}$, and (\ref{2.21.1}),
we can obtain for sufficiently small $\vep$ that
\begin{align*}
\|T e_{n+1}^\d\| &\le \left(3 \mu +\gamma^{\frac{1}{p-1}}\right)
\left(\theta \a_{n+1} \|\xi_0-\xi^\dag\|^{p^*} \right)^{\frac{1}{p}}.
\end{align*}
We therefore obtain the desired estimates (\ref{est1}) and (\ref{est2}).

Finally we show that $n_\d\le \hat{n}_\d$. We first claim that for $0\le n\le \hat{n}_\d$ there holds
\begin{align*}
\| y^\d-y-T e_n^\d\| & \le \d + \left(\gamma^{\frac{1}{p-1}} \theta^{\frac{1}{p}}+ C \vep \right)
\|\xi_0-\xi^\dag\|^{\frac{1}{p-1}} \a_n^{\frac{1}{p}}.
\end{align*}
In fact, for $n=0$ this inequality follows from (\ref{1.3}) and (\ref{9.15.2}), and for $1\le n\le \hat{n}_\d$
it follows from (\ref{2.20.5}), (\ref{est2})  and (\ref{1.5}).
Therefore, by using Assumption \ref{A1} and the estimates (\ref{est1}) and (\ref{est2}), we can obtain
\begin{align}\label{9.15.3}
\|y^\d-F(x_n^\d)\| & \le \|y^\d-y -T e_n^\d\| +\|y-F(x_n^\d)+ T e_n^\d\| \nonumber\\
& \le \d + \left(\gamma^{\frac{1}{p-1}} \theta^{\frac{1}{p}} + C \vep \right)
\|\xi_0-\xi^\dag\|^{\frac{1}{p-1}} \a_n^{\frac{1}{p}}.
\end{align}
If $\hat{n}_\d=0$, then $\a_0\le \mu^{-p} \d^p \|\xi_0-\xi^\dag\|^{-p^*}$. Therefore
\begin{align*}
\|F(x_0)-y^\d\| \le \d +\left(\gamma^{\frac{1}{p-1}} \theta^{\frac{1}{p}} \mu^{-1} +C\vep\right) \d.
\end{align*}
In view of (\ref{gamma}) we have for sufficiently small $\vep$ that $\|F(x_0)-y^\d\|\le \tau \d$.
Consequently $n_\d=0$.

In the following we assume that $\hat{n}_\d\ge 1$. Observing from (\ref{1.5}) and (\ref{nhat}) that
for $n=\hat{n}_\d$ and $\hat{n}_\d-1$ there holds
$$
\a_n^{\frac{1}{p}} \le \left(\theta \a_{\hat{n}_\d}\right)^{\frac{1}{p}}
\le \mu^{-1}\theta^{\frac{1}{p}} \|\xi_0-\xi^\dag\|^{-\frac{1}{p-1}} \d.
$$
Thus, from (\ref{9.15.3}) we have for $n=\hat{n}_\d$ and $\hat{n}_\d-1$ that
\begin{align*}
\|y^\d-F(x_n^\d)\| & \le (1+C\vep) \d + \gamma^{\frac{1}{p-1}} \theta^{\frac{2}{p}} \mu^{-1} \d.
\end{align*}
Since $\mu$ is chosen to satisfy (\ref{gamma}), we have for sufficiently small $\vep$ that
$$
\|y^\d-F(x_n^\d)\|\le \tau \d \quad \mbox{for } n=\hat{n}_\d \mbox{ and } \hat{n}_\d-1.
$$
Therefore, by the definition of $n_\d$ we have $n_\d\le \hat{n}_\d$. \hfill $\Box$
\end{proof}

\begin{remark}\label{remark3.1}
We use $\Theta_F$ in (\ref{IRGN}) to guarantee that $\{x_n^\d\}\subset D(F)$ without assuming
$x^\dag$ is an interior point of $D(F)$. If $x^\dag$ is an interior point of $D(F)$ so that
$B_\rho(x^\dag)\subset D(F)$ for a ball $B_\rho(x^\dag)$ of radius $\rho>0$, we can replace $\Theta_F$ in (\ref{IRGN})
by $\Theta$ and define $x_{n+1}^\d$ to be the unique minimizer of the convex minimization problem
\begin{equation*}
\min_{x\in \X} \left\{\|y^\d-F(x_n^\d)-F'(x_n^\d) (x-x_n^\d)\|^p +\a_n D_{\xi_0} \Theta(x, x_0)\right\}.
\end{equation*}
The same argument in the proof of Lemma \ref{L3.1.1} can be used to show that for sufficiently small $\|\xi_0-\xi^\dag\|$ there
holds $x_n^\d\in B_\rho(x^\dag)\subset D(F)$ for all $0\le n\le \hat{n}_\d$. Therefore, the modified method
is well-defined and all the results in this paper still hold.

\end{remark}

As a byproduct of the estimates in Lemma \ref{L3.1.1}, we can prove a weak convergence result of our method.

\begin{theorem}\label{T4.5.1}
Assume that  the conditions in Lemma \ref{L3.1.1} hold. Assume also that $\X$ is reflexive and $F$ is weakly closed.
If the method (\ref{IRGN}) is terminated by either Rule \ref{Rule1}, \ref{Rule2}, or \ref{Rule3} with $\tau>1$,
then for any sequence $\{y^{\d_k}\}$ satisfying $\|y^{\d_k}-y\|\le \d_k$ with $\d_k\rightarrow 0$ as $k\rightarrow \infty$,
$\{x_{n_{\d_k}}^{\d_k}\}$ has a subsequence that converges weakly in $\X$ to a solution of (\ref{1.1}) in
$\overline{B_\rho(x^\dag)}\cap D(F)$. If $x^\dag$ is the unique solution of (\ref{1.1}) in
$\overline{B_\rho(x^\dag)}\cap D(F)$, then $x_{n_\d}^\d$ converges weakly in $\X$ to $x^\dag$ as $\d\rightarrow 0$.
\end{theorem}

\begin{proof}
It follows from Lemma \ref{L3.1.1} that $\{x_{n_{\d_k}}^{\d_k}\}\subset B_\rho(x^\dag)\cap D(F)$.
Since $\X$ is reflexive, $\{x_{n_{\d_k}}^{\d_k}\}$ has a
subsequence that converges weakly in $\X$ to some $\bar{x}\in \X$. By using the weak lower semi-continuity of
norms in Banach spaces and the convexity and closedness of $D(F)$, we have $\bar{x} \in \overline{B_\rho(x^\dag)}\cap D(F)$.
Moreover, since $\|F(x_{n_{\d_k}}^{\d_k})-y^{\d_k}\|\le 2\tau \d_k$, we have $\|F(x_{n_{\d_k}}^{\d_k})-y\|\rightarrow 0$
as $k\rightarrow \infty$. By the weakly closedness of $F$ we have $F(\bar{x})=y$, i.e. $\bar{x}$ is a solution of
(\ref{1.1}) in $\overline{B_\rho(x^\dag)}\cap D(F)$. \hfill $\Box$
\end{proof}

\begin{remark}
In Theorem \ref{T4.5.1} we only obtain the weak convergence. The proof of strong convergence remains open in general.
However, in section \ref{Sect5} we will prove a strong convergence result when $\Y$ is a Hilbert space and $\Theta$ is a
$2$-convex function. Moreover, in some situations we are interested in the strong convergence in a Banach space ${\mathcal Z}$ in
which $\X$ can be compactly embedded, the weak convergence in $\X$ is already enough for the purpose.
\end{remark}

\section{\bf Rates of convergence}\label{Sect4}
\setcounter{equation}{0}

In this section we will derive rates of convergence for $x_{n_\d}^\d$ to $x^\dag$ under certain source
conditions. In Hilbert space setting, the usual source conditions are
\begin{equation}\label{s1}
x_0-x^\dag =(T^*T)^{\frac{\nu}{2}} \omega
\end{equation}
for some $0<\nu\le 1$ and $\omega\in \X$. By the interpolation inequality it is easy to see that
(\ref{s1}) implies
\begin{equation}\label{s2}
\l x_0-x^\dag, x-x^\dag\r \le \|\omega\| \|x-x^\dag\|^{1-\nu} \|T(x-x^\dag)\|^\nu, \quad \forall x\in \X.
\end{equation}
In Banach space setting, the formulation (\ref{s1}) for source conditions does not make sense in general.
However, we may use (\ref{s2}) to propose the replacement of the form
$$
\l \xi_0-\xi^\dag, x-x^\dag\r \le \beta \|x-x^\dag\|^{1-\nu} \|T(x-x^\dag)\|^\nu, \quad \forall x\in \X.
$$
Considering the $p$-convexity of $\Theta_F$, we may further modify this into the form
\begin{equation}\label{source}
\l \xi_0-\xi^\dag, x-x^\dag\r \le \beta \left[D_{\xi^\dag} \Theta_F(x, x^\dag)\right]^{\frac{1-\nu}{p}} \|T(x-x^\dag)\|^\nu,
\quad \forall x\in \X
\end{equation}
for some $0<\nu\le 1$ and $\beta \ge 0$. We therefore obtain source conditions formulated as
variational inequalities, whose analog have already been introduced in \cite{KH2010}. We will use (\ref{source})
as our source conditions to derive convergence rates.

\begin{theorem}\label{T3.1.1}
Let $\X$ and $\Y$ be Banach spaces, let $\Theta:\X\to (-\infty, \infty]$ be a proper, lower
semi-continuous, $p$-convex function for some $p\ge 2$, let $\{\a_n\}$ be a sequence
satisfying (\ref{1.5}), and let $F$ satisfy Assumption \ref{A1}.
If the source condition (\ref{source}) is satisfied with $0<\nu\le 1$ and if
$\vep:=(K_0+K_1)\|\xi_0-\xi^\dag\|^{\frac{1}{p-1}}$ is sufficiently small, then for the integer $n_\d$
determined by either Rule \ref{Rule2} or Rule \ref{Rule3} with $\tau>1$ there holds
$$
D_{\xi^\dag} \Theta_F(x_{n_\d}^\d, x^\dag)\le C \beta^{\frac{p}{p-1+\nu}} \d^{\frac{p\nu}{p-1+\nu}}
$$
and thus
$$
\|x_{n_\d}^\d-x^\dag\| \le C \beta^{\frac{1}{p-1+\nu}} \d^{\frac{\nu}{p-1+\nu}},
$$
where $C$ is a constant depending only on $p$, $\gamma$, $\theta$, $\tau$ and $\nu$.
\end{theorem}

We will complete the proof of Theorem \ref{T3.1.1} by proving a series of lemmas.

\begin{lemma}\label{L3.7.1}
Under the same conditions in Theorem \ref{T3.1.1}, if the source condition (\ref{source}) holds
and $\vep:=(K_0+K_1)\|\xi_0-\xi^\dag\|^{\frac{1}{p-1}}$ is sufficiently small, then there holds
\begin{align*}
\|T(x_n^\d-x^\dag)\| \le C \left(\d +\left(\beta \a_n^{\frac{p-1+\nu}{p}}\right)^{\frac{1}{p-1}}\right)
\end{align*}
for all $0\le n\le \hat{n}_\d$, where $\hat{n}_\d$ is the integer defined by (\ref{nhat}).
\end{lemma}

\begin{proof}
We will use (\ref{2.20.1}). In view of the Young's inequality, it follows from (\ref{source}) that
\begin{align*}
\l  \xi&_0-\xi^\dag, x_{n+1}^\d-x^\dag\r
 \le \frac{1-\nu}{p} D_{\xi^\dag} \Theta_F(x_{n+1}^\d, x^\dag)
 +\frac{p-1+\nu}{p} \left(\beta  \|T e_{n+1}^\d\|^\nu\right)^{\frac{p}{p-1+\nu}}.
\end{align*}
Plugging this into (\ref{2.20.1}) gives
\begin{align}\label{3.7.2}
& \|  y^\d -F(x_n^\d)-F'(x_n^\d) (x_{n+1}^\d-x_n^\d)\|^p +\frac{p-1+\nu}{p} \a_n D_{\xi^\dag} \Theta_F(x_{n+1}^\d, x^\dag) \nonumber\\
& \le \|y^\d-F(x_n^\d)-F'(x_n^\d) (x^\dag-x_n^\d)\|^p
+\frac{p-1+\nu}{p} \a_n \left(\beta \|T e_{n+1}^\d\|^\nu\right)^{\frac{p}{p-1+\nu}}.
\end{align}
This inequality implies immediately that
\begin{align}\label{3.7.3}
\| y^\d  -F(x_n^\d)-F'(x_n^\d) (x_{n+1}^\d-x_n^\d)\|
& \le \|y^\d-F(x_n^\d)-F'(x_n^\d) (x^\dag-x_n^\d)\| \nonumber\\
& \quad \, + \a_n^{\frac{1}{p}} \left(\beta \|T e_{n+1}^\d\|^\nu\right)^{\frac{1}{p-1+\nu}}.
\end{align}
In view of (\ref{3.8.20}), we can obtain from (\ref{3.7.3}) that
\begin{align*}
\|y^\d-y -T e_{n+1}^\d \| & \le \d + \|(T-F'(x_n^\d)) e_{n+1}^\d\|
+ \a_n^{\frac{1}{p}} \left(\beta \|T e_{n+1}^\d\|^\nu\right)^{\frac{1}{p-1+\nu}}\\
&\quad \, + 2\|y-F(x_n^\d)- F'(x_n^\d) (x^\dag-x_n^\d)\|.
\end{align*}
With the help of Assumption \ref{A1} we then obtain
\begin{align}\label{3.7.4}
\|  y^\d  -y -T e_{n+1}^\d \| & \le \d + 3(K_0+K_1) \|e_n^\d\| \|T e_n^\d\|  + K_0 \|e_n^\d \| \|T e_{n+1}^\d\| \nonumber\\
& \quad \, + K_1\|e_{n+1}^\d\| \|T e_n^\d\|  + \a_n^{\frac{1}{p}} \left(\beta \|T e_{n+1}^\d\|^\nu\right)^{\frac{1}{p-1+\nu}}.
\end{align}
By employing the estimate on $\|e_n^\d\|$ from Lemma \ref{L3.1.1}, we can obtain from (\ref{3.7.4}) that
\begin{align*}
\|T e_{n+1}^\d \| & \le 2 \d + C \vep \|T e_n^\d\|  + C \vep  \|T e_{n+1}^\d\|
+ \a_n^{\frac{1}{p}} \left(\beta \|T e_{n+1}^\d\|^\nu\right)^{\frac{1}{p-1+\nu}}.
\end{align*}
By using the Young's inequality again we can derive that
\begin{align*}
\|T e_{n+1}^\d \| & \le 2 \d + C \vep \|T e_n^\d\|  + \left(\frac{\nu}{p-1+\nu} +C \vep\right) \|T e_{n+1}^\d\| \\
& \quad \, + \frac{p-1}{p-1+\nu} \left(\beta \a_n^{\frac{p-1+\nu}{p}} \right)^{\frac{1}{p-1}}.
\end{align*}
Therefore if $\vep$ is sufficiently small, then we can obtain
\begin{align*}
\|T e_{n+1}^\d \| & \le \frac{3p}{p-1} \d + C \vep \|T e_n^\d\|
+ 2 \left( \beta \a_n^{\frac{p-1+\nu}{p}} \right)^{\frac{1}{p-1}}.
\end{align*}
Thus, in view of $\a_n\le \theta \a_{n+1}$, if we further assume that $\vep$ is sufficiently small, then
an induction argument would show that
$$
\|T e_n^\d\| \le \frac{4 p}{p-1} \d +  3 \left(\beta (\theta\a_n)^{\frac{p-1+\nu}{p}} \right)^{\frac{1}{p-1}}
$$
for all $0\le n\le \hat{n}_\d$ if we could show that this is also true for $\|T e_0\|$. Observing that
\begin{align}\label{9.15.5}
D_{\xi^\dag} \Theta_F(x_0, x^\dag)
&\le D_{\xi^\dag}\Theta_F(x_0, x^\dag) + D_{\xi_0}\Theta_F(x^\dag, x_0) = \l \xi_0-\xi^\dag, x_0-x^\dag\r \nonumber\\
&\le \beta \left[D_{\xi^\dag} \Theta_F(x_0, x^\dag)\right]^{\frac{1-\nu}{p}} \|T e_0\|^\nu.
\end{align}
This implies that $D_{\xi^\dag} \Theta_F(x_0, x^\dag) \le \left(\beta \|T e_0\|^\nu\right)^{\frac{p}{p-1+\nu}}$
and consequently by the $p$-convexity of $\Theta$ we have
$$
\|x_0-x^\dag\|\le \gamma (\beta\|T e_0\|^\nu)^{\frac{1}{p-1+\nu}}.
$$
Therefore
$$
\|T e_0\|\le \|T\| \|e_0\| \le \gamma \|T\| \left(\beta \|T e_0\|^\nu\right)^{\frac{1}{p-1+\nu}}.
$$
In view of $\|T\|\le \a_0^{\frac{1}{p}}/\gamma$ we can obtain
\begin{equation}\label{3.8.3}
\|T e_0\| \le \left( \gamma \|T\| \beta^{\frac{1}{p-1+\nu}}\right)^{\frac{p-1+\nu}{p-1}}
\le \left(\beta \a_0^{\frac{p-1+\nu}{p}}\right)^{\frac{1}{p-1}}.
\end{equation}
We therefore complete the proof. \hfill $\Box$
\end{proof}

\begin{lemma}\label{L3.8.2}
Under the same conditions in Theorem \ref{T3.1.1}, if $\vep:=(K_0+K_1)\|\xi_0-\xi^\dag\|^{\frac{1}{p-1}}$
is sufficiently small, then there holds
\begin{equation}\label{3.8.1}
\|y^\d-y-T e_n^\d\| \le \left(\frac{\tau+1}{2} +C\vep \right) \d
+ C \left(\beta \a_n^{\frac{p-1+\nu}{p}} \right)^{\frac{1}{p-1}}
\end{equation}
for all $0\le n\le \hat{n}_\d$.
\end{lemma}

\begin{proof}
We will use (\ref{3.7.4}). In view of the estimates on $\|e_n^\d\|$ given in Lemma \ref{L3.1.1}, we can
obtain from (\ref{3.7.4}) that
\begin{align*}
\|y^\d-y-T e_{n+1}^\d\| \le \d +C\vep \|Te_n^\d\| +C\vep \|T e_{n+1}^\d\|
+ \a_n^{\frac{1}{p}} \left(\beta \|T e_{n+1}^\d\|^\nu\right)^{\frac{1}{p-1+\nu}}.
\end{align*}
Using the estimates on $\|T e_n^\d\|$ in Lemma \ref{L3.7.1}, the fact $\a_n \le \theta \a_{n+1}$ and the
inequality $(a+b)^t\le a^t+b^t$ for $a, b\ge 0$ and $0\le t \le 1$, we have
\begin{align*}
\|y^\d-y-T e_{n+1}^\d\| & \le \d +C \vep \left(\d + \left(\beta \a_{n+1}^{\frac{p-1+\nu}{p}}\right)^{\frac{1}{p-1}}\right) \\
&\quad \, + C \a_{n+1}^{\frac{1}{p}} \beta^{\frac{1}{p-1+\nu}}
\left(\d +\left(\beta \a_{n+1}^{\frac{p-1+\nu}{p}}\right)^{\frac{1}{p-1}}\right)^{\frac{\nu}{p-1+\nu}}\\
& \le (1+ C\vep ) \d +  C \left(\beta \a_{n+1}^{\frac{p-1+\nu}{p}} \right)^{\frac{1}{p-1}}\\
&\quad \, + C \a_{n+1}^{\frac{1}{p}} \beta^{\frac{1}{p-1+\nu}} \d^{\frac{\nu}{p-1+\nu}}.
\end{align*}
By using the Young's inequality we have
$$
C \a_{n+1}^{\frac{1}{p}} \beta^{\frac{1}{p-1+\nu}} \d^{\frac{\nu}{p-1+\nu}}
\le \frac{\tau-1}{2} \d + C' \left(\beta \a_{n+1}^{\frac{p-1+\nu}{p}} \right)^{\frac{1}{p-1}}.
$$
Combining the above two estimates we therefore obtain (\ref{3.8.1}) for $1\le n\le \hat{n}_\d$.
It remains only to check (\ref{3.8.1}) for $n=0$. By using $\|y^\d-y\|\le \d$ and (\ref{3.8.3}),
this is obvious. \hfill $\Box$
\end{proof}

\begin{lemma}\label{L8.3.5}
Under the same conditions in Theorem \ref{T3.1.1}, there exists a positive universal constant $c_1$ such that
$$
\a_n \ge c_1 \left(\frac{\d^{p-1}}{\beta}\right)^{\frac{p}{p-1+\nu}}
$$
for all $0\le n<n_\d$, where $n_\d$ is the integer defined by either Rule \ref{Rule2} or Rule \ref{Rule3} with $\tau>1$.
\end{lemma}

\begin{proof}
If $n_\d=1$ we must have $\|F(x_0)-y^\d\|>\tau \d$. It then follows from Assumption \ref{A1} and (\ref{3.8.3}) that
\begin{align*}
(\tau-1) \d \le \|F(x_0)-y\| \le (1+C\vep) \|T e_0\| \le C\left(\beta \a_0^{\frac{p-1+\nu}{p}}\right)^{\frac{1}{p-1}}.
\end{align*}
This implies the desired estimate on $\a_0$. So we may assume that $n_\d\ge 2$. From the definition of $n_\d$ we
have for $1\le n<n_\d$ that
\begin{equation}\label{3.10.1}
\tau \d \le \max\left\{\|F(x_n^\d)-y^\d\|, \|F(x_{n-1}^\d)-y^\d\|\right\}.
\end{equation}
By using Lemma \ref{L3.8.2}, Assumption \ref{A1}, and the estimates in Lemma \ref{L3.1.1} we have
for all $0\le n\le \hat{n}_\d$ that
\begin{align*}
\|F(x_n^\d)-y^\d\| &\le \|y^\d-y-T e_n^\d\| +\|F(x_n^\d)-y-T e_n^\d\| \\
& \le \left(\frac{\tau+1}{2} +C \vep\right) \d +C \left(\beta \a_n^{\frac{p-1+\nu}{p}} \right)^{\frac{1}{p-1}}
+ C \vep \|T e_n^\d\|.
\end{align*}
In view of the estimate on $\|T e_n^\d\|$ in Lemma \ref{L3.7.1}, it follows for $0\le n\le \hat{n}_\d$ that
$$
\|F(x_n^\d)-y^\d\| \le \left(\frac{\tau+1}{2} +C \vep\right) \d
+C \left(\beta \a_n^{\frac{p-1+\nu}{p}} \right)^{\frac{1}{p-1}}.
$$
Recall that $n_\d\le \hat{n}_\d$ and $\a_n\le \a_{n-1}\le \theta \a_n$, we therefore obtain from (\ref{3.10.1}) that
$$
\tau \d \le \left(\frac{\tau+1}{2} +C \vep\right) \d +C \left(\beta \a_n^{\frac{p-1+\nu}{p}} \right)^{\frac{1}{p-1}},
\quad 0\le n<n_\d.
$$
Thus, if $\vep$ is sufficiently small, then we can derive that
$$
\d \le C \left(\beta \a_n^{\frac{p-1+\nu}{p}} \right)^{\frac{1}{p-1}}, \quad 0\le n<n_\d
$$
which gives the conclusion immediately. \hfill $\Box$
\end{proof}

Finally we prove Theorem \ref{T3.1.1} concerning the convergence rates of the method.

\begin{proof}{\it of Theorem \ref{T3.1.1}}.
We first consider the case $n_\d\ge 1$. Then for $1\le n\le n_\d$ we have from (\ref{3.7.2}) that
\begin{align*}
\a_{n-1}  D_{\xi^\dag} \Theta_F(x_n^\d, x^\dag)
& \le \frac{p}{p-1+\nu} \|y^\d-F(x_{n-1}^\d) -F'(x_{n-1}^\d) (x^\dag-x_{n-1}^\d)\|^p \\
&\quad \, + \a_{n-1} \left(\beta \|T e_n^\d\|^\nu\right)^{\frac{p}{p-1+\nu}}.
\end{align*}
Therefore, by using Assumption \ref{A1}, the estimate on $\|e_n^\d\|$ in Lemma \ref{L3.1.1},
the inequality $(a+b)^t\le 2^{t-1} (a^t+b^t)$ for $a, b\ge 0$ and $t\ge 1$,
we can obtain
\begin{align}\label{3.8.10}
D_{\xi^\dag} \Theta_F(x_n^\d, x^\dag) \le \frac{p2^{p-1}}{(p-1)\a_{n-1}} \left(\d^p + C \vep^p \|T e_{n-1}^\d\|^p\right)
+ \left(\beta \|T e_n^\d\|^\nu\right)^{\frac{p}{p-1+\nu}}.
\end{align}
Observing that Assumption \ref{A1} and the estimate on $\|e_n^\d\|$ in Lemma \ref{L3.1.1} imply
$$
\|T e_n^\d\| \le \|F(x_n^\d)-y\| +C \vep \|T e_n^\d\|.
$$
Thus, if $\vep$ is sufficiently small, then we have $\|T e_n^\d\| \le 2 \|F(x_n^\d)-y\|$. Since $n_\d$ is
determined by Rule \ref{Rule2} or Rule \ref{Rule3}, we have
$$
\|F(x_{n_\d}^\d)-y^\d\|+ \|F(x_{n_\d-1}^\d)-y^\d\|  \le 2\tau \d.
$$
We therefore obtain
$$
\|T e_{n_\d}^\d\|+ \|T e_{n_\d-1}^\d\|  \le 4(1+\tau) \d.
$$
Now we can take $n=n_\d$ in (\ref{3.8.10}) to obtain
$$
D_{\xi^\dag} \Theta_F(x_{n_\d}^\d, x^\dag) \le C\left( \frac{\d^p}{\a_{n_\d-1}} +\left(\beta \d^\nu\right)^{\frac{p}{p-1+\nu}}\right).
$$
An application of Lemma \ref{L8.3.5} then gives the desired rates of convergence.

For the case $n_\d=0$, we have $\|F(x_0)-y^\d\|\le \tau \d$ and thus $\|T e_0\|\le 2(1+\tau) \d$.
We may use (\ref{9.15.5}) to derive that
$$
D_{\xi^\dag}\Theta_F(x_0, x^\dag) \le \left(\beta \|T e_0\|^\nu\right)^{\frac{p}{p-1+\nu}}
\le C \left(\beta \d^\nu\right)^{\frac{p}{p-1+\nu}}.
$$
This completes the proof. \hfill $\Box$
\end{proof}

\begin{remark} The similar argument can be applied to derive the rate of convergence
under the general source condition
$$
\l \xi_0-\xi^\dag, x-x^\dag\r
\le \beta \left[D_{\xi^\dag} \Theta_F(x, x^\dag)\right]^{\frac{1}{p}} f\left(\frac{\|T(x-x^\dag)\|^p}{D_{\xi^\dag} \Theta_F(x, x^\dag)}\right)
$$
for some index function $f$ with suitable properties.
\end{remark}

\section{\bf Convergence}\label{Sect5}
\setcounter{equation}{0}

Although Theorem \ref{T3.1.1} gives the rates of convergence, it does not tell whether the method
is convergent when the source condition is not known to be satisfied.
In this section we will consider the situation that $\X$ is a reflexive Banach space, $\Y$ is a Hilbert space,
and $\Theta$ is a proper, lower semi-continuous,  $2$-convex function satisfying (\ref{pconv}) with $p=2$,
and derive the convergence result without assuming any source condition.
We will use $(\cdot, \cdot)$ to denote the inner product in $\Y$. In this situation, $x_{n+1}^\d$ is
the unique minimizer of the convex minimization problem
\begin{equation}\label{3.14.1}
\min_{x\in \X} \left\{\|y^\d-F(x_n^\d)-F'(x_n^\d) (x-x_n^\d)\|^2 +\a_n D_{\xi_0} \Theta_F(x,x_0)\right\},
\end{equation}
where $\Theta_F$ is the proper, lower semi-continuous, convex function on $\X$ defined by (\ref{theta_F}) satisfying
$$
\|z-x\|\le \gamma \left[D_\xi \Theta_F(z,x)\right]^{\frac{1}{2}},
\qquad \forall z\in \X, \, x\in D(\p \Theta_F) \mbox{ and } \xi\in \p \Theta_F(x).
$$
Let $n_\d$ be the integer determined by either Rule \ref{Rule1}, Rule \ref{Rule2} or Rule \ref{Rule3} with $\tau>1$. We will show that
$x_{n_\d}^\d\rightarrow x^\dag$ as $\d\rightarrow 0$ if
\begin{equation}\label{3.14.10}
\xi_0-\xi^\dag\in {\mathcal N}(T)^\perp,
\end{equation}
where ${\mathcal N}(T):=\{x\in \X: T x=0\}$ denotes the null space of $T$ and
$$
{\mathcal N}(T)^\perp:=\{\xi \in \X^*: \l \xi, x\r =0 \mbox{ for all } x\in {\mathcal N}(T)\}.
$$

We will derive the convergence result in two steps. In the first step, we consider the noise-free
iterative sequence $\{x_n\}$ defined by (\ref{3.14.1}) with $y^\d$ replaced by $y$, i.e. $x_{n+1}$
is the unique minimizer of the problem
\begin{equation}\label{3.14.2}
\min_{x\in \X} \left\{\|y-F(x_n)-F'(x_n) (x-x_n)\|^2 +\a_n D_{\xi_0} \Theta_F(x,x_0)\right\}.
\end{equation}
We will show that $x_n\rightarrow x^\dag$ as $n\rightarrow \infty$. In the second step, we will
consider the relation between $x_n^\d$ and $x_n$ and establish some crucial stability estimates.
The definition of $n_\d$ then enables us to derive the desired convergence result.

In order to achieve these two steps, we need the following simple result which plays a crucial role
in the arguments.

\begin{lemma} \label{L3.11.3}
Assume that $\X$ is a Banach space and $\Y$ is a Hilbert space. Let $A$ and $\hat{A}$ be two bounded
linear operators from $\X$ to $\Y$. For $\a>0$ let $x_\a$ be the minimizer of the problem
\begin{equation}\label{2.15.1}
\min_{x\in \X} \left\{\|y-A x\|^2 +\a D_{\xi_0} \Theta_F(x, x_0)\right\},
\end{equation}
and let $\hat{x}_\a$ be the minimizer of (\ref{2.15.1}) with $A$, $y$, $x_0$ and $\xi_0$ replaced by
$\hat{A}$, $\hat{y}$, $\hat{x}_0$ and $\hat{\xi}_0\in \p \Theta(\hat{x}_0)$ respectively. Then there holds
\begin{align*}
\|\hat{y} -y-\hat{A}(\hat{x}_\a-x_\a)\|^2 +\a D_{\xi_\a} \Theta_F(\hat{x}_\a, x_\a)
&\le \|\hat{y}-y\|^2 +\a \l \hat{\xi}_0-\xi_0, \hat{x}_\a-x_\a\r \\
& \quad \, + 2 ((\hat{A}-A) x_\a, \hat{A} (x_\a-\hat{x}_\a) )\\
& \quad \, + 2 (y-A x_\a, (A-\hat{A})(x_\a-\hat{x}_\a)),
\end{align*}
where $\xi_\a:=\xi_0 +\frac{2}{\a} A^* (y-A x_\a)\in \p \Theta_F(x_\a)$.
\end{lemma}

\begin{proof}
Since $x_\a$ is the minimizer of (\ref{2.15.1}), we immediately have $\xi_\a\in \p \Theta_F(x_\a)$.
By using the minimizing property of $\hat{x}_\a$, we have
$$
\|\hat{y}-\hat{A} \hat{x}_\a\|^2 +\a D_{\hat{\xi}_0} \Theta_F(\hat{x}_\a, \hat{x}_0)
\le \|\hat{y}-\hat{A} x_\a\|^2 +\a D_{\hat{\xi}_0} \Theta_F(x_\a, \hat{x}_0).
$$
Recall that
$$
D_{\hat{\xi}_0} \Theta_F(\hat{x}_\a, \hat{x}_0)-D_{\hat{\xi}_0} \Theta_F(x_\a, \hat{x}_0)
=D_{\xi_\a} \Theta_F(\hat{x}_\a, x_\a) +\l \xi_\a-\hat{\xi}_0, \hat{x}_\a-x_\a\r
$$
and
\begin{align*}
\|\hat{y}-\hat{A} \hat{x}_\a\|^2 &=\|y-\hat{A} x_\a\|^2 +2 (y-\hat{A} x_\a, \hat{y}-y -\hat{A}(\hat{x}_\a-x_\a))\\
&\quad \,  +\|\hat{y}-y -\hat{A} (\hat{x}_\a-x_\a)\|^2.
\end{align*}
Combining the above three equations we can derive that
\begin{align*}
 \|  \hat{y} -y  -\hat{A}(\hat{x}_\a-x_\a)\|^2 +\a D_{\xi_\a} \Theta_F(\hat{x}_\a, x_\a)
& \le \a \l \hat{\xi}_0 -\xi_\a, \hat{x}_\a-x_\a\r \\
&-2 (y-\hat{A} x_\a, \hat{y}-y -\hat{A} (\hat{x}_\a-x_\a) )\\
& +\|\hat{y}-\hat{A} x_\a\|^2 -\|y-\hat{A} x_\a\|^2.
\end{align*}
Since
$$
\|\hat{y}-\hat{A} x_\a\|^2 -\|y-\hat{A} x_\a\|^2=\|\hat{y}-y\|^2 +2(\hat{y}-y, y-\hat{A} x_\a),
$$
we can obtain
\begin{align*}
 \|  \hat{y} -y  -\hat{A}(\hat{x}_\a-x_\a)\|^2 +\a D_{\xi_\a} \Theta_F(\hat{x}_\a, x_\a)
 &\le  \a \l \hat{\xi}_0-\xi_\a, \hat{x}_\a-x_\a\r +\|\hat{y}-y\|^2 \\
 & + 2(y-\hat{A} x_\a, \hat{A} (\hat{x}_\a-x_\a)).
\end{align*}
In view of the fact $\a(\xi_0-\xi_\a)+2 A^*(y-A x_\a)=0$, by rearranging the terms we therefore
obtain the desired result. \hfill $\Box$
\end{proof}

\subsection{Convergence of the noise-free iterations}

In this subsection we will show for the noise-free iteration $\{x_n\}$ that $x_n\rightarrow x^\dag$ as
$n\rightarrow \infty$ if $\xi_0-\xi^\dag$ satisfies (\ref{3.14.10}). We first confirm this convergence
result under the stronger condition
$$
\xi_0-\xi^\dag =T^* \omega
$$
for some $\omega\in \Y^*$. This is included in the following result.

\begin{lemma}\label{L3.14.2}
Assume that $\X$ is a Banach space, $\Y$ is a Hilbert space, and $\Theta: \X\to (-\infty, \infty]$
is a proper, lower semi-continuous, $2$-convex function. Let $F$ satisfy Assumption \ref{A1} and let $\{\a_n\}$ satisfy (\ref{1.5}).
If $\xi_0-\xi^\dag=T^* \omega$ for some $\omega\in \Y^*$ and $(K_0+K_1) \|\xi_0-\xi^\dag\|$ is sufficiently small,
then for all $n$ there hold
\begin{align*}
\|x_n- x^\dag\| \le C \|\omega\| \a_n^{1/2} \quad \mbox{and} \quad  \|T(x_n-x^\dag)\| \le C \|\omega\| \a_n.
\end{align*}
\end{lemma}

\begin{proof}
Since $\xi_0-\xi^\dag=T^* \omega$, the source condition (\ref{source}) holds with $\nu=1$
and $\beta=\|\omega\|$. Thus we can apply Lemma \ref{L3.7.1} to obtain the estimate on $\|T(x_n-x^\dag)\|$ immediately.
In order to derive the estimate on $\|x_n-x^\dag\|$, we use (\ref{3.8.10}) which can be formulated as
$$
D_{\xi^\dag} \Theta_F(x_n, x^\dag) \le \frac{\|T (x_{n-1}-x^\dag)\|^2}{\a_{n-1}} +\|\omega\| \|T (x_n-x^\dag)\|.
$$
By using the estimates on $\|T (x_n-x^\dag)\|$, (\ref{1.5}) and the $2$-convexity of $\Theta_F$, we can obtain the
desired estimate. \hfill $\Box$
\end{proof}

In order to derive convergence under merely the condition (\ref{3.14.10}), we will use the following strategy.
We first find $\hat{x}_0\in D(\p \Theta_F):= D(F)\cap D(\p \Theta)$ and $\hat{\xi}_0\in \p \Theta_F(\hat{x}_0)$ such
that $\hat{\xi}_0$ is sufficiently close to $\xi_0$ and $\hat{\xi}_0-\xi^\dag\in \R(T^*)$, where $\R(T^*)$
denotes the range of $T^*$. We then use $\hat{x}_0$ and $\hat{\xi}_0$ as new initial data and define $\{\hat{x}_n\}$
by letting $\hat{x}_{n+1}$ be the unique minimizer of the problem
$$
\min_{x\in \X} \left\{\|y-F(\hat{x}_n)-F'(\hat{x}_n) (x-\hat{x}_n)\|^2 +\a_n D_{\hat{\xi}_0}\Theta_F(x, \hat{x}_0)\right\}.
$$
According to Lemma \ref{L3.14.2}, we have $\hat{x}_n \rightarrow x^\dag$ as $n\rightarrow \infty$.
In order to pass this convergence result to $\{x_n\}$, we need a perturbation result on
$\{x_n\}$ with respect to $\xi_0$.

\begin{lemma}\label{L3.15.1}
Assume that $\X$ is a Banach space, $\Y$ is a Hilbert space, and $\Theta: \X\to (-\infty, \infty]$
is a proper, lower semi-continuous, $2$-convex function. Let $F$ satisfy Assumption \ref{A1} and
let $\{\a_n\}$ satisfy (\ref{1.5}). If
$$
\vep:=(K_0+K_1) \max\{\|\xi_0-\xi^\dag\|, \|\hat{\xi}_0 - \xi^\dag\|\}
$$
is sufficiently small, then for all $n$ there hold
\begin{align}\label{3.29.2}
\|x_n- \hat{x}_n\|  \le 2 \gamma^2 \|\xi_0-\hat{\xi}_0\| \quad \mbox{and} \quad
\|T(x_n-\hat{x}_n)\| \le 2 \gamma \theta^{1/2} \a_n^{1/2} \|\xi_0-\hat{\xi}_0\|.
\end{align}

\end{lemma}

\begin{proof}
Using the same argument in the proof of Lemma \ref{L3.1.1}, it follows that if $\vep$ is sufficiently small
then $x_n$ and $\hat{x}_n$ are well-defined for all $n$ and there hold the estimates
\begin{align}\label{3.29.3}
\|e_n\| +\frac{\|T e_n\|}{\sqrt{\a_n}} &\le C \|\xi_0-\xi^\dag\|,\quad
\|\hat{e}_n\| +\frac{\|T \hat{e}_n\|}{\sqrt{\a_n}} \le C \|\hat{\xi}_0-\xi^\dag\|,
\end{align}
where
$$
e_n:=x_n-x^\dag \qquad \mbox{and} \qquad \hat{e}_n:=\hat{x}_n-x^\dag.
$$

In the following we will prove (\ref{3.29.2})
by induction. Since $x_0=\nabla \Theta_F^*(\xi_0)$ and $\hat{x}_0= \nabla \Theta_F^*(\hat{\xi}_0)$, we have
from (\ref{3.29.1}) and the scaling condition $\|T\|\le \a_0^{1/2}/\gamma$ that (\ref{3.29.2}) holds
for $n=0$. Now we assume that (\ref{3.29.2}) holds for some $n$ and show that it also holds
true for $n+1$.

Let $\tilde{\Theta}(x):=\Theta_F(x+x^\dag)$. Then $\xi_0\in \p \tilde{\Theta}(x_0-x^\dag)$ and
$\tilde{\Theta}$ is still a 2-convex function. By using the definition of $x_{n+1}$, it is easy to see
that $e_{n+1}:=x_{n+1}-x^\dag$ is the minimizer of the minimization problem
\begin{equation}\label{2.15.200}
\min_{e\in \X} \left\{ \|g_n-F'(x_n) e\|^2 +\a_n D_{\xi_0} \tilde{\Theta}(e, x_0-x^\dag)\right\},
\end{equation}
where
\begin{equation}\label{2.15.11}
g_n:=y-F(x_n)-F'(x_n) (x^\dag-x_n).
\end{equation}
Similarly, $\hat{e}_{n+1}:=\hat{x}_{n+1}-x^\dag$ is the unique minimizer of the minimization problem
\begin{equation}\label{2.15.2}
\min_{e\in \X} \left\{ \|\hat{g}_n-F'(\hat{x}_n) e\|^2 +\a_n D_{\hat{\xi}_0} \tilde{\Theta}(e, \hat{x}_0-x^\dag)\right\},
\end{equation}
where
$$
\hat{g}_n:=y-F(\hat{x}_n)-F'(\hat{x}_n) (x^\dag-\hat{x}_n).
$$
Let $T_n:=F'(x_n)$ and $\hat{T}_n:=F'(\hat{x}_n)$. It then follows from Lemma \ref{L3.11.3}
and the $2$-convexity of $\tilde{\Theta}$ that
\begin{align*}
&\| \hat{g}_n-g_n  -\hat{T}_n (\hat{x}_{n+1} - x_{n+1})\|^2 +\a_n \left(\gamma^{-1}\|\hat{x}_{n+1}- x_{n+1}\|\right)^2 \nonumber \\
& \le \|\hat{g}_n-g_n\|^2 +\a_n \l \hat{\xi}_0-\xi_0, \hat{x}_{n+1}-x_{n+1}\r
+2 ((\hat{T}_n-T_n) e_{n+1}, \hat{T}_n (x_{n+1}-\hat{x}_{n+1})) \nonumber\\
& +2 (g_n-T_n e_{n+1}, (T_n-\hat{T}_n) (x_{n+1}-\hat{x}_{n+1})).
\end{align*}
In view of the identity $\|a+b\|^2 =\|a\|^2 +2 (a, b) +\|b\|^2$ in Hilbert spaces, we can write
\begin{align*}
\| \hat{g}_n & -g_n  -\hat{T}_n (\hat{x}_{n+1} - x_{n+1})\|^2 \\
& = \|\hat{g}_n-g_n\|^2 -2 (\hat{g}_n-g_n, \hat{T}_n (\hat{x}_{n+1}-x_{n+1}))
+\|T(\hat{x}_{n+1}-x_{n+1})\|^2 \\
& + 2 (T(\hat{x}_{n+1}-x_{n+1}), (\hat{T}_n-T) (\hat{x}_{n+1}-x_{n+1}))\\
& + \|(\hat{T}_n-T) (\hat{x}_{n+1}-x_{n+1})\|^2.
\end{align*}
Therefore we can obtain
\begin{align}\label{3.10.2}
\|T(\hat{x}_{n+1} -x_{n+1})\|^2 +\a_n \left(\gamma^{-1} \|\hat{x}_{n+1}-x_{n+1}\|\right)^2
& \le \a_n \l \hat{\xi}_0-\xi_0, \hat{x}_{n+1}-x_{n+1}\r \nonumber\\
& \quad \,  + I_1 +I_2 +I_3 +I_4,
\end{align}
where
\begin{align*}
I_1 &= 2 ((\hat{T}_n-T_n) e_{n+1}, \hat{T}_n (x_{n+1}-\hat{x}_{n+1})),\\
I_2 &= 2 (g_n-T_n e_{n+1}, (T_n-\hat{T}_n) (x_{n+1}-\hat{x}_{n+1})), \\
I_3 &= 2 (\hat{g}_n-g_n, \hat{T}_n (\hat{x}_{n+1}-x_{n+1})), \\
I_4 &=  2 (T(x_{n+1}-\hat{x}_{n+1}), (\hat{T}_n-T) (\hat{x}_{n+1}-x_{n+1})).
\end{align*}
In the following we will estimate $I_j$ for $j=1, \cdots, 4$. With the help of Assumption \ref{A1},
(\ref{1.5}), (\ref{3.29.3}) and the induction hypotheses, we can derive that
\begin{align*}
& \|T_n e_{n+1}\| +\|g_n-T_n e_{n+1}\| \le C \a_n^{1/2} \|\xi_0-\xi^\dag\|,\\
&\|\hat{T}_n (x_n-\hat{x}_n)\|  \le C \a_n^{1/2} \|\xi_0-\hat{\xi}_0\|, \quad
\|(\hat{T}_n-T_n) e_{n+1} \| \le C \vep\a_n^{1/2} \|\xi_0-\hat{\xi}_0\|,\\
& \|(\hat{T}_n -T) (x_{n+1}-\hat{x}_{n+1})\| \le C\vep \|T(x_{n+1}-\hat{x}_{n+1})\| + C \vep \a_n^{1/2} \|x_{n+1}-\hat{x}_{n+1}\|,\\
& \|\hat{T}_n (x_{n+1}-\hat{x}_{n+1})\| \le (1+C\vep) \|T(x_{n+1}-\hat{x}_{n+1})\| + C \vep \a_n^{1/2} \|x_{n+1}-\hat{x}_{n+1}\|
\end{align*}
and
\begin{align*}
\|&(T_n-\hat{T}_n) (x_{n+1}-\hat{x}_{n+1})\| \\
& \le C(K_0+K_1) \|\xi_0-\hat{\xi}_0\| \left(\|T(x_{n+1}-\hat{x}_{n+1})\|
 +\a_n^{1/2} \|x_{n+1}-\hat{x}_{n+1}\|\right).
\end{align*}
Moreover, by writing
$$
\hat{g}_n-g_n=\left(F(x_n)-F(\hat{x}_n)-\hat{T}_n (x_n-\hat{x}_n)\right) +(\hat{T}_n-T_n) e_n,
$$
we can use Assumption \ref{A1}, (\ref{3.29.3}), and the induction hypotheses to derive that
$$
\|\hat{g}_n-g_n\| \le C\vep \a_n^{1/2} \|\xi_0-\hat{\xi}_0\|.
$$
By making use of the above estimates we therefore obtain
\begin{align*}
|I_1|+|I_2|+|I_3| \le C\vep \|\xi_0-\hat{\xi}_0\|
\left( \a_n^{1/2} \|T(x_{n+1}-\hat{x}_{n+1})\| +\a_n \|x_{n+1}-\hat{x}_{n+1}\|\right)
\end{align*}
and
$$
|I_4|\le C\vep \|T(x_{n+1}-\hat{x}_{n+1})\|^2 +C\vep \a_n \|x_{n+1}-\hat{x}_{n+1}\|^2.
$$
Combining these estimates on $I_j$, $j=1,\cdots, 4$ with (\ref{3.10.2}) gives
\begin{align*}
\|T(x_{n+1} & -\hat{x}_{n+1})\|^2 +\a_n \left(\gamma^{-1} \|x_{n+1} - \hat{x}_{n+1}\|\right)^2 \\
&\le C\vep \|\xi_0-\hat{\xi}_0\| \left( \a_n^{1/2} \|T(x_{n+1}-\hat{x}_{n+1})\| +\a_n \|x_{n+1}-\hat{x}_{n+1}\|\right)\\
& +C\vep \|T(x_{n+1}-\hat{x}_{n+1})\|^2 +C\vep \a_n \|x_{n+1}-\hat{x}_{n+1}\|^2\\
& + \a_n \|\xi_0-\hat{\xi}_0\| \|x_{n+1}-\hat{x}_{n+1}\|.
\end{align*}
Therefore, if $\vep$ is sufficiently small, we can obtain immediately that
\begin{align*}
\|T(x_{n+1} & -\hat{x}_{n+1})\|^2 +\a_n \left(\gamma^{-1} \|x_{n+1} - \hat{x}_{n+1}\|\right)^2
\le 4 \gamma^2 \a_n \|\xi_0-\hat{\xi}_0\|^2.
\end{align*}
In view of the condition $\a_n\le \theta \a_{n+1}$, we therefore obtain the desired estimates. \hfill $\Box$
\end{proof}

Now we are ready to prove the convergence of the noise-free iteration $\{x_n\}$.

\begin{theorem}\label{P3.12.1}
Let $\X$ be a reflexive Banach space and $\Y$ be a Hilbert space, let $\Theta$ be a proper, lower semi-continuous,
$2$-convex function on $\X$. Let $F$ satisfy Assumption \ref{A1} and let $\{\a_n\}$ satisfy (\ref{1.5}).
If $\xi_0-\xi^\dag\in \N(T)^\perp$ and $(K_0+K_1) \|\xi_0-\xi^\dag\|$ is sufficiently small, then  there hold
$$
\lim_{n\rightarrow \infty} \|x_n-x^\dag\| =0 \quad \mbox{and} \quad
\lim_{n\rightarrow \infty} \frac{\|T(x_n-x^\dag)\|}{\sqrt{\a_n}}=0.
$$
\end{theorem}

\begin{proof}
Let $\Theta_F^*$ denote the Fenchel conjugate of $\Theta_F$. It is known that $D(\Theta_F^*)=\X^*$,
$\Theta_F^*$ is Fr\'{e}chet differentiable and its gradient $\nabla \Theta_F^*: \X^*\to \X$ satisfies
$$
\|\nabla \Theta_F^*(\xi)-\nabla \Theta_F^*(\eta) \|\le \gamma^2\|\xi-\eta\|, \quad \forall \xi, \eta\in \X^*.
$$

Let $0<\epsilon<\|\xi_0-\xi^\dag\|$ be sufficiently small. Since $\X$ is reflexive, we have
$\N(T)^\perp =\overline{\R(T^*)}$. Therefore $\xi_0-\xi^\dag\in \overline{\R(T^*)}$. Consequently,
we can choose $\hat{\xi}_0\in \X^*$ such that $\|\xi_0-\hat{\xi}_0\|\le\epsilon$ and
$\hat{\xi}_0-\xi^\dag\in \R(T^*)$. We now define $\hat{x}_0:=\nabla \Theta_F^*(\hat{\xi}_0)$. Then we have
$\hat{x}_0\in D(\p \Theta_F)$ and $\hat{\xi}_0 \in \p \Theta_F(\hat{x}_0)$. Moreover
$$
\|\hat{x}_0-x_0\| =\|\nabla \Theta_F^*(\hat{\xi}_0)-\nabla \Theta_F^*(\xi_0)\| \le \gamma^2 \|\hat{\xi}_0-\xi_0\|
\le \gamma^2 \epsilon.
$$
Since $x_0\in B_\rho(x^\dag)$, by taking $\epsilon>0$ to be small enough, we can guarantee that
$\hat{x}_0\in B_\rho(x^\dag) \cap D(\p \Theta_F)$. We then use this $\hat{x}_0$ as
an initial guess to define $\{\hat{x}_n\}$ as above. Since the smallness of $(K_0+K_1)\|\xi_0-\xi^\dag\|$
implies the smallness of $(K_0+K_1) \|\hat{\xi}_0 -\xi^\dag\|$, we may use Lemma \ref{L3.15.1} to conclude
that there is a constant $C_*$ independent of $n$ such that
$$
\|x_n-\hat{x}_n\| +\frac{\|T(x_n-\hat{x}_n)\|}{\sqrt{\a_n}} \le C_* \|\xi_0-\hat{\xi}_0\|\le C_* \epsilon,
\quad \forall n.
$$
On the other hand, since $\hat{\xi}_0-\xi^\dag\in \R(T^*)$, it follows from Lemma \ref{L3.14.2} and (\ref{1.5})
that there exists an integer $n_0$ such that
$$
\|\hat{x}_n-x^\dag\| +\frac{\|T(\hat{x}_n-x^\dag)\|}{\sqrt{\a_n}} \le \epsilon, \quad \forall n\ge n_0.
$$
Consequently
$$
\|x_n-x^\dag\| +\frac{\|T(x_n-x^\dag)\|}{\sqrt{\a_n}} \le (1+C_*)\epsilon, \quad \forall n\ge n_0.
$$
Since $\epsilon>0$ can be arbitrarily small, we therefore obtain the convergence result. \hfill $\Box$
\end{proof}

\subsection{Main convergence result}

Although we have shown in the previous subsection the convergence of the noise-free iteration $\{x_n\}$
as $n\rightarrow \infty$, our ultimate aim is to show that $x_{n_\d}^\d\rightarrow x^\dag$ as $\d\rightarrow 0$
with the integer $n_\d$ defined by either Rule \ref{Rule1}, \ref{Rule2}, or \ref{Rule3} with $\tau>1$.
We still need some stability estimates contained in the following result.

\begin{lemma}\label{L3.15.10}
Assume that all the conditions with $p=2$ in Lemma \ref{L3.1.1} hold, and assume also that
$\Y$ is a Hilbert space.  If $\vep:=(K_0+K_1) \|\xi_0-\xi^\dag\|$
is sufficiently small, then for all $0\le n\le \hat{n}_\d$ there hold
$$
\|x_n^\d-x_n\| \le 3 (1+\gamma)^2 \frac{\d}{\sqrt{\a_n}} \quad \mbox{and}
\quad \|F(x_n^\d)-F(x_n)-y^\d+y\|\le (1+C\vep ) \d,
$$
where $\hat{n}_\d$ is the integer defined by (\ref{nhat}).
\end{lemma}

\begin{proof}
We first prove by induction that
\begin{equation}\label{3.31.1}
\|x_n^\d-x_n\| \le 3(1+\gamma)^2 \frac{\d}{\sqrt{\a_n}} \,\,\, \mbox{ and } \,\,\,
\|T(x_n-x_n^\d)\| \le 3(1+\gamma) \d, \quad 0\le n\le \hat{n}_\d.
\end{equation}
Since $x_0^\d=x_0$, the estimates are trivial for $n=0$. We now assume that the estimates
are true for some $n<\hat{n}_\d$ and show that they are also true for $n+1$. We will use the similar argument
in the proof of Lemma \ref{L3.15.1}. By the definition of $x_{n+1}^\d$, it is easy to see that
$e_{n+1}^\d:=x_{n+1}^\d-x^\dag$ is the unique minimizer of the problem
$$
\min_{e\in \X} \left\{\|g_n^\d-F'(x_n^\d) e\|^2 +\a_n D_{\xi_0} \tilde{\Theta}(e, x_0-x^\dag)\right\},
$$
where
$$
g_n^\d:=y^\d-F(x_n^\d) -F'(x_n^\d) (x^\dag-x_n^\d).
$$
Recall that $e_{n+1}:=x_{n+1}-x^\dag$ is the unique minimizer of the problem (\ref{2.15.200}) with
$g_n$ given by (\ref{2.15.11}). In view of Lemma \ref{L3.11.3} and the $2$-convexity of $\tilde{\Theta}$, we can obtain
\begin{align}\label{3.11.10}
\|g_n^\d  & - g_n -F'(x_n^\d) (x_{n+1}^\d-x_{n+1})\|^2 +\a_n \left(\gamma^{-1} \|x_{n+1}^\d- x_{n+1}\|\right)^2 \nonumber\\
& \quad \le \|g_n^\d-g_n\|^2 + 2\left((F'(x_n^\d)-F'(x_n)) e_{n+1}, F'(x_n^\d) (x_{n+1}-x_{n+1}^\d)\right) \nonumber\\
& \quad + 2 \left(g_n-F'(x_n) e_{n+1}, (F'(x_n)-F'(x_n^\d)) (x_{n+1}-x_{n+1}^\d)\right).
\end{align}
We can write
\begin{align*}
\|g_n^\d & -g_n -F'(x_n^\d) (x_{n+1}^\d-x_{n+1})\|^2  \\
& =\|g_n^\d-g_n\|^2 -2 \left(g_n^\d-g_n, F'(x_n^\d) (x_{n+1}^\d-x_{n+1})\right) +\|T(x_{n+1}^\d-x_{n+1})\|^2\\
& + 2 \left(T (x_{n+1}^\d-x_{n+1}), (F'(x_n^\d)-T) (x_{n+1}^\d-x_{n+1}) \right) \\
& + \|(F'(x_n^\d)-T) (x_{n+1}^\d-x_{n+1})\|^2.
\end{align*}
Therefore, it follows from (\ref{3.11.10}) that
\begin{align}\label{4.13.1}
\|T (x_{n+1}^\d-x_{n+1})\|^2 +\a_n \left(\gamma^{-1} \|x_{n+1}^\d- x_{n+1}\|\right)^2
\le J_1 +J_2 +J_3 +J_4,
\end{align}
where
\begin{align*}
J_1 & = 2\left((F'(x_n^\d)-F'(x_n)) e_{n+1}, F'(x_n^\d) (x_{n+1}-x_{n+1}^\d)\right),\\
J_2 & = 2 \left(g_n-F'(x_n) e_{n+1}, (F'(x_n)-F'(x_n^\d)) (x_{n+1}-x_{n+1}^\d)\right),\\
J_3 & = 2 \left(g_n^\d-g_n, F'(x_n^\d) (x_{n+1}^\d-x_{n+1})\right), \\
J_4 & = 2 \left(T (x_{n+1}-x_{n+1}^\d), (F'(x_n^\d)-T) (x_{n+1}^\d-x_{n+1}) \right).
\end{align*}
In the following we will estimate $J_j$ for $j=1, \cdots, 4$. With the help of Assumption \ref{A1},
(\ref{1.5}), (\ref{3.29.3}), the estimates in Lemma \ref{L3.1.1} and the induction hypotheses,
we can derive that
\begin{align*}
& \|F'(x_n) e_n^\d\| +\|F'(x_n) e_{n+1}\| +\|g_n-F'(x_n) e_{n+1}\| \le C \a_n^{1/2} \|\xi_0-\xi^\dag\|,\\
&\|F'(x_n) (x_n^\d-x_n) \| \le C \d, \qquad
\|(F'(x_n^\d) - F'(x_n) ) e_{n+1} \| \le C \vep \d,\\
& \|(F'(x_n^\d) -T) (x_{n+1}^\d -x_{n+1})\| \le C\vep \|T(x_{n+1}^\d- x_{n+1})\| + C \vep \a_n^{1/2} \|x_{n+1}^\d-x_{n+1}\|,\\
& \|F'(x_n) (x_{n+1}^\d-x_{n+1})\| \le (1+C\vep) \|T(x_{n+1}^\d-x_{n+1})\| + C \vep \a_n^{1/2} \|x_{n+1}^\d-x_{n+1}\|,\\
& \|F'(x_n^\d) (x_{n+1}^\d-x_{n+1})\| \le (1+C\vep) \|T(x_{n+1}^\d-x_{n+1})\| + C \vep \a_n^{1/2} \|x_{n+1}^\d-x_{n+1}\|
\end{align*}
and
\begin{align*}
\|(F'(x_n)& -F'(x_n^\d)) (x_{n+1}-x_{n+1}^\d)\| \\
& \le C(K_0+K_1) \d \left(\a_n^{-1/2} \|T(x_{n+1}^\d-x_{n+1})\|
 +\|x_{n+1}^\d-x_{n+1}\|\right).
\end{align*}
In order to estimate $\|g_n^\d-g_n\|$, we use the expressions of $g_n^\d$ and $g_n$ to write
\begin{align*}
g_n^\d-g_n &= y^\d-y -\left[F(x_n^\d)-F(x_n)-F'(x_n) (x_n^\d-x_n)\right] + \left[F'(x_n^\d)-F'(x_n)\right] e_n^\d.
\end{align*}
By using Assumption \ref{A1}, the estimates in Lemma \ref{L3.1.1}, (\ref{3.29.3}) and the induction hypotheses,
we can derive that
\begin{align}\label{3.31.2}
\|g_n^\d-g_n-y^\d+y\| & \le C\vep \d.
\end{align}
Therefore
\begin{equation}\label{3.31.3}
\|g_n^\d-g_n\| \le (1+C\vep ) \d.
\end{equation}
By making use of the above estimates we therefore obtain
\begin{align*}
& |J_1|+|J_2| \le C\vep \d \left( \|T(x_{n+1}^\d-x_{n+1})\| +\a_n^{1/2} \|x_{n+1}^\d-x_{n+1}\|\right) \\
& |J_3| \le (2+C\vep) \d \left( \|T(x_{n+1}^\d-x_{n+1})\| +\a_n^{1/2} \|x_{n+1}^\d-x_{n+1}\|\right)\\
& |J_4| \le C\vep \left(\|T(x_{n+1}^\d-x_{n+1})\|^2 + \a_n \|x_{n+1}^\d-x_{n+1}\|^2\right).
\end{align*}
Combining the above estimates on $J_j$ for $j=1, \cdots, 4$ we therefore obtain from (\ref{4.13.1}) that
\begin{align*}
\|T(x_{n+1}^\d&-x_{n+1})\|^2 +\a_n \left(\gamma^{-1} \|x_{n+1}^\d- x_{n+1}\|\right)^2 \\
& \le (2+C\vep) \d \left( \|T(x_{n+1}^\d-x_{n+1})\| +\a_n^{1/2} \|x_{n+1}^\d-x_{n+1}\|\right)\\
& + C\vep \left(\|T(x_{n+1}^\d-x_{n+1})\|^2 + \a_n \|x_{n+1}^\d-x_{n+1}\|^2\right).
\end{align*}
Thus, if $\vep$ is sufficiently small, we have
\begin{align*}
\|T(x_{n+1}^\d &-x_{n+1})\|^2 +\a_n \left(\gamma^{-1} \|x_{n+1}^\d- x_{n+1}\|\right)^2  \le 9(1+\gamma)^2 \d^2.
\end{align*}
This together with $\a_{n+1}\le \a_n$ completes the proof of (\ref{3.31.1}).

By using the estimate (\ref{3.31.1}) we have $|J_1|+|J_2|\le C\vep \d^2$. Thus, we may use (\ref{3.11.10})
and (\ref{3.31.3}) to obtain
\begin{equation}\label{9.17.1}
\|g_n^\d-g_n-F'(x_n^\d) (x_{n+1}^\d-x_{n+1})\|\le (1+C\vep ) \d.
\end{equation}
Observing that Assumption \ref{A1}, Lemma \ref{L3.1.1}, and (\ref{3.31.1}) imply
\begin{equation}\label{9.17.2}
\|(T-F'(x_n^\d)) (x_{n+1}^\d-x_{n+1})\| \le C \vep \d.
\end{equation}
We may use (\ref{3.31.2}), (\ref{9.17.1}) and (\ref{9.17.2}) to obtain
$$
\|y^\d-y -T(x_n^\d-x_n)\|\le (1+C\vep) \d, \qquad 0\le n\le \hat{n}_\d
$$
since it is trivial for $n=0$ because $x_0^\d=x_0$.

Finally, we can use Assumption \ref{A1}, Lemma \ref{L3.1.1}, (\ref{3.29.3}) and (\ref{3.31.1}) to derive that
\begin{align*}
\|F(x_n^\d)-F(x_n)-y^\d+y\| &\le \|F(x_n^\d)-F(x_n)-F'(x_n) (x_n^\d-x_n)\| \\
& \quad \, + \|[F'(x_n)-T](x_n^\d-x_n)\|\\
&\quad \, +\|y^\d-y-T(x_n^\d-x_n)\|\\
&\le (1+C\vep ) \d.
\end{align*}
The proof is therefore complete. \hfill $\Box$
\end{proof}

Now we are ready to prove the main convergence result.

\begin{theorem}\label{T5.2}
Let $\X$ be a reflexive Banach space and $\Y$ be a Hilbert space, and let
$\Theta: \X \to (-\infty, \infty]$ be a proper, lower semi-continuous, $2$-convex function.
Let $F$ satisfy Assumption \ref{A1} and let $\{\a_n\}$ satisfy (\ref{1.5}). Assume that $x^\dag$ is
the unique solution of (\ref{1.1}) in $B_\rho(x^\dag)\cap D(F)$. If $\xi_0-\xi^\dag\in \N(T)^\perp$ and
$\vep:=(K_0+K_1) \|\xi_0-\xi^\dag\|$ is sufficiently small, then for the method (\ref{IRGN}) terminated
by either Rule \ref{Rule1}, \ref{Rule2}, or \ref{Rule3} with $\tau>1$ there holds $x_{n_\d}^\d \rightarrow
x^\dag$ as $\d\rightarrow 0$.
\end{theorem}

\begin{proof}

We complete the proof by considering two cases. Assume first that there is a sequence
$\{y^{\d_k}\}$ satisfying $\|y^{\d_k}-y\|\le \d_k$ with $\d_k\rightarrow 0$ such that
$n_k:=n_{\d_k}$ converges to a finite integer $n$ as $k\rightarrow \infty$. We may assume that
$n_k=n$ for all $k$. By Lemma \ref{L3.15.10} we have $x_{n_k}^{\d_k}=x_n^{\d_k}\rightarrow x_n$
as $k\rightarrow \infty$. Since the definition of $n_k$ implies
$$
\|F(x_{n_k}^{\d_k})-y^{\d_k}\|\le 2 \tau \d_k,
$$
by taking $k\rightarrow \infty$ we can obtain $F(x_n)=y$. Since $x^\dag$ is the unique solution of
(\ref{1.1}) in $B_\rho(x^\dag)$, we have $x_n=x^\dag$ and hence $x_{n_k}^{\d_k}\rightarrow x^\dag$
as $k\rightarrow \infty$.

Assume next that there is a sequence $\{y^{\d_k}\}$ satisfying $\|y^{\d_k}-y\|\le \d_k$ with $\d_k\rightarrow 0$
such that $n_k:=n_{\d_k}\rightarrow \infty$ as $k\rightarrow \infty$. By the first estimate in Lemma
\ref{L3.15.10} we have
$$
\|x_{n_k}^{\d_k}-x^\dag\| \le 3(1+\gamma)^2 \frac{\d_k}{\sqrt{\a_{n_k}}} +\|x_{n_k}-x^\dag\|.
$$
By using the definition of $n_k$ and the second estimate in Lemma \ref{L3.15.10} we can obtain
\begin{align*}
\tau \d_k &\le \max\left\{\|F(x_{n_k-1}^{\d_k})-y^{\d_k}\|, \|F(x_{n_k-2}^{\d_k})-y^{\d_k}\|\right\}\\
&\le (1+C\vep ) \d_k + \max\left\{ \|F(x_{n_k-1})-y\|, \|F(x_{n_k-2})-y\|\right\}.
\end{align*}
By using Assumption \ref{A1} and (\ref{3.29.3}) we can show that
$\|F(x_n)-y\| \le 2\|T(x_n-x^\dag)\|$ for all $n$ if $\vep$ is sufficiently small, and consequently
$$
\d_k \le \frac{8}{\tau-1} \max\left\{\|T e_{n_k-1}\|, \|T e_{n_k-2}\|\right\}.
$$
Since $n_k\rightarrow \infty$, it follows from Theorem \ref{P3.12.1} and (\ref{1.5}) that $\d_k/\sqrt{\a_{n_k}}
\rightarrow 0$ as $k\rightarrow \infty$. Moreover, Theorem \ref{P3.12.1} also implies that
$\|x_{n_k}-x^\dag\| \rightarrow 0$ as $k\rightarrow \infty$. We therefore obtain again $\|x_{n_k}^{\d_k}-x^\dag\|
\rightarrow 0$ as $k\rightarrow \infty$. \hfill $\Box$
\end{proof}

\section{Applications to parameter identification problems}\label{Sect6}

In this section we consider some examples on parameter identification in partial differential equations
to illustrate that Assumption \ref{A1}(d) can be verified for a wide range of applications. We also report
some numerical experiments to test the efficiency of our method.

\begin{example}\label{e6.1}
We first consider the identification of the parameter $c$ in the boundary value
problem
\begin{equation}\label{8.27.1}
\left\{\begin{array}{ll} -\triangle u+c u=f \qquad& \mbox{in } \Omega\\
u=g \qquad& \mbox{on } \partial \Omega
\end{array}\right.
\end{equation}
from an $L^2(\Omega)$-measurement of the state $u$, where $\Omega\subset
{\mathbb R}^N, N\le 3,$ is a bounded domain with Lipschitz boundary
$\partial \Omega$, $f\in L^2(\Omega)$ and $g\in H^{3/2}(\partial
\Omega)$. We assume $c^\dag\in L^2(\Omega)$ is the sought solution.
This problem reduces to solving an equation of the form
(\ref{1.1}) if we define the nonlinear operator $F$ to be the
parameter-to-solution mapping
$$
F: L^2(\Omega)\to L^2(\Omega), \qquad F(c):=u(c)
$$
with $u(c)\in H^2(\Omega)\subset L^2(\Omega)$ being
the unique solution of (\ref{8.27.1}). Such $F$ is well-defined on
$$
D(F):=\left\{c\in L^2(\Omega): \|c-\hat{c}\|_{L^2}\le \gamma_0
\mbox{ for some } \hat{c}\ge 0 \mbox{ a.e.}\right\}
$$
for some positive constant $\gamma_0>0$. It is well known that $F$
has Fr\'{e}chet derivative
\begin{equation}\label{8.27.2}
F'(c) h=-A(c)^{-1}(hF(c)), \qquad h\in L^2(\Omega),
\end{equation}
where $A(c): V:=H^2\cap H_0^1\to L^2$ is defined
by $A(c)u:=-\triangle u+c u$ which is an isomorphism uniformly in a ball
$B_\rho(c^\dag)\subset D(F)$ around $c^\dag$. Let $V'$ be the dual space of $V:=H^2\cap H_0^1$ with
respect to the bilinear form
\begin{equation}\label{bilinear}
\langle \varphi, \psi\rangle =\int_\Omega \varphi(x) \psi(x) dx.
\end{equation}
Then $A(c)$ extends to an isomorphism from $L^2(\Omega)$ to $V'$. Since (\ref{8.27.2}) implies for any $c,
d\in B_\rho(c^\dag)$ and $h\in L^2(\Omega)$
$$
\left(F'(c)-F'(d)\right) h=-A(c)^{-1}\left((c-d)F'(d) h\right)
-A(c)^{-1}\left(h(F(c)-F(d))\right),
$$
and since $L^1(\Omega)$ embeds into $V'$ due to the restriction $N\le 3$, we have
\begin{align}\label{8.27.3}
\|(F'(c)& -F'(d))h\|_{L^2} \nonumber\\
&\le \|A(c)^{-1}\left((c-d)F'(d)
h\right)\|_{L^2}+\|A(c)^{-1}\left(h(F(c)-F(d))\right)\|_{L^2}\nonumber\\
&\le C \|(c-d)F'(d)h\|_{V'}+C\|h(F(c)-F(d))\|_{V'} \nonumber\\
&\le C \|(c-d)F'(d)h\|_{L^1}+C\|h(F(c)-F(d))\|_{L^1}\nonumber\\
&\le C \|c-d\|_{L^2}\|F'(d)h\|_{L^2}+ C\|F(c)-F(d)\|_{L^2}
\|h\|_{L^2}.
\end{align}
On the other hand, observing that
$$
F(c)-F(d)=-A(d)^{-1}\left((c-d)F(c)\right),
$$
by using (\ref{8.27.2}) we have
$$
F(c)-F(d)-F'(d)(c-d)=-A(d)^{-1}\left((c-d)\left(F(c)-F(d)\right)\right).
$$
Thus, by a similar argument as above,
$$
\|F(c)-F(d)-F'(d)(c-d)\|_{L^2} \le C
\|c-d\|_{L^2}\|F(c)-F(d)\|_{L^2}.
$$
Therefore, if $\rho>0$ is small enough, we have
$\|F(c)-F(d)\|_{L^2}\le 2 \|F'(d)(c-d)\|_{L^2}$ for
$c, d\in B_\rho(c^\dag)$, which together
with (\ref{8.27.3}) verifies Assumption \ref{A1}(d).

In order to reconstruct $c^\dag$, we choose a $2$-convex function $\Theta$, an initial guess
$c_0\in D(F)\cap D(\p \Theta)$ and $\xi_0\in \p \Theta(c_0)$. By adopting Remark \ref{remark3.1},
we then define $c_{n+1}^\d$ as the minimizer of the convex optimization problem
\begin{equation}\label{min6.2}
\min_{c\in D(F)}\left\{\int_\Omega |u^\d-F(c_n^\d)-F'(c_n^\d) (c-c_n^\d)|^2 dx +\a_n D_{\xi_0}\Theta(c,c_0)\right\}.
\end{equation}
Let $n_\d$ be the integer determined by either Rule \ref{Rule1}, \ref{Rule2}, or \ref{Rule3} with $\tau>1$.
Then, by Theorem \ref{T5.2}, we have $\|c_{n_\d}^\d-c^\dag\|_{L^2(\Omega)} \rightarrow 0$ as $\d\rightarrow 0$.

In the following we present two numerical experiments for this example to test our method. In these computation,
we always choose $\Theta$ to be nonnegative with $\Theta(0)=0$ so that we can take $c_0=0$ and $\xi_0=0$ and consequently
$D_{\xi_0} \Theta(c, c_0)=\Theta(c)$.

\begin{figure}[ht!]
     \begin{center}
%
            \includegraphics[width=1\textwidth, height=2.6in]{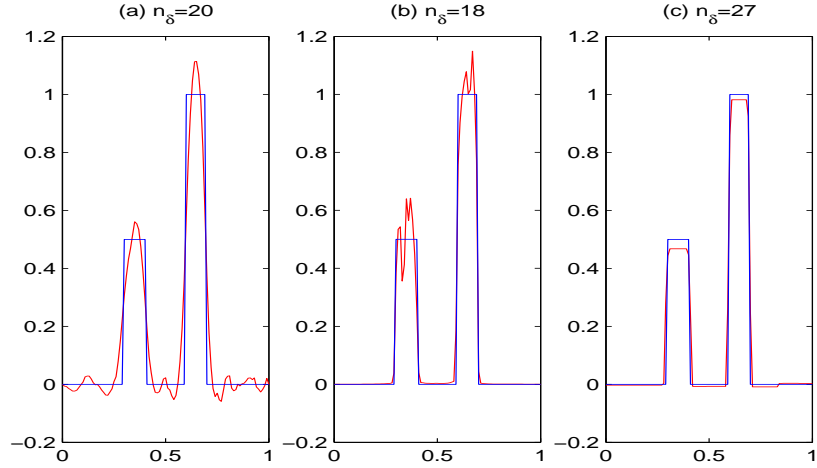}
%
    \end{center}
    \caption{Numerical results for the one-dimensional problem in Example \ref{e6.1} with
    different choices of $\Theta$: (a) $\Theta(c)=\|c\|_{L^2}^2$; (b) $\Theta(c)=\lambda \|c\|_{L^2}^2 +\|c\|_{L^1}$
    with $\lambda=0.01$; (c) $\Theta(c)=\lambda \|c\|_{L^2}^2 +\int_{[0,1]} |D c|$ with $\lambda=0.01$.%
             }%
    \label{fig6.2}
\end{figure}

In the first numerical experiment we consider the one-dimensional problem over the interval $\Omega=(0,1)$ with the sought
solution given by
$$
c^\dag(t)=\left\{\begin{array}{lll}
0.5, & \quad  \mbox{if } 0.3\le t\le 0.4,\\
1.0, & \quad \mbox{if } 0.6\le t\le 0.7,\\
0, & \quad \mbox{elsewhere}.
\end{array}\right.
$$
We assume that the inhomogeneous term is $f(t)=(1+5t) c^\dag(t)$ and the boundary data are
$u(0)=1$ and $u(1)=6$. Then $u(c^\dag)=1+5t$. In our computation, instead of $u(c^\dag)$ we use random noise
data $u^\d$ satisfying $\|u^\d-u(c^\dag)\|_{L^2[0,1]}=\d$ with noisy level $\d>0$; we take $\d=0.1\times 10^{-3}$
and $\a_n=2^{-n}$. The differential equations involved are solved approximately by a finite difference method by dividing $[0,1]$
into $100$ subintervals of equal length with the resulting tridiagonal system solved by the Thomas algorithm.
The convex optimization problems (\ref{min6.2}) is solved by a restart conjugate gradient method (\cite{SY2006}). The iteration
is terminated by Rule \ref{Rule1}, i.e. the discrepancy principle, with $\tau=1.05$. In Figure \ref{fig6.2} we
report the computational results with different choices of $\Theta$. In (a) we report the result with $\Theta(c)=\|c\|_{L^2}^2$
for which the corresponding method becomes the iteratively regularized Gauss--Newton method in Hilbert spaces.
Although the reconstruction tells something on the sought solution, it does not tell more information such
as sparsity, discontinuities and constancy since the result is too oscillatory. In (b) we report the result corresponding to
$\Theta(c)=\lambda \|c\|_{L^2}^2 +\|c\|_{L^1}$ with $\lambda=0.01$.  Since $\|c\|_{L^1}$ is non-smooth, we replace it by
$\int_0^1 \sqrt{|c|^2+\varepsilon}$ with $\varepsilon=10^{-6}$ in our computation. It is clear that the sparsity
of the sought solution is significantly reconstructed. The reconstruction result, however, is still oscillatory
on the nonzero parts which is typical for this choice of $\Theta$. In (c) we report the result corresponding to $\Theta(c)=
\lambda \|c\|_{L^2}^2 +\int_{[0,1]} |D c|$ with $\lambda=0.01$. Again we replace $\int_{[0,1]} |Dc|$ by
$\int_{[0,1]} \sqrt{|D c|^2 +\varepsilon}$ with $\varepsilon=10^{-6}$. The reconstruction is rather satisfactory and
the notorious oscillatory effect is efficiently removed.

\begin{figure}[ht!]
     \begin{center}
%
            \includegraphics[width=1\textwidth, height= 4in]{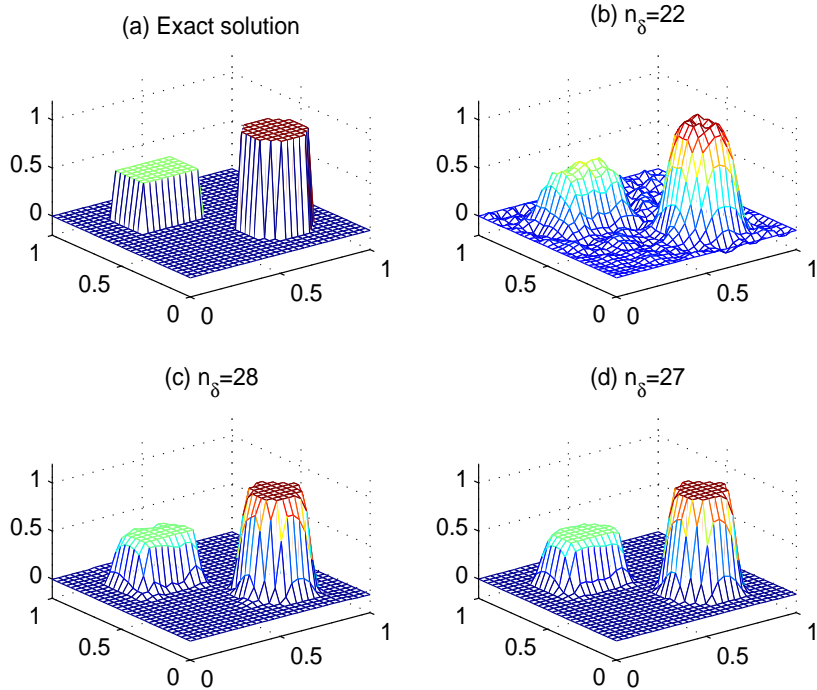} 
%
    \end{center}
    \caption{Numerical results for the two-dimensional problem in Example \ref{e6.1}: (a) exact solution;
    (b) $\Theta(c)=\|c\|_{L^2}^2$; (c) and (d) $\Theta(c) = \lambda \|c\|_{L^2}^2 +\int_\Omega |Dc|$ with $\lambda=0.01$
    and $\la=1.0$ respectively. %
             }%
     \label{fig6.3}
\end{figure}

In the second numerical experiment we consider the two dimensional problem with $\Omega=[0,1]\times[0,1]$.
The sought solution is
$$
c^\dag(x,y)=\left\{\begin{array}{lll}
1,  & \quad  \mbox{if } (x-0.3)^2+(y-0.7)^2\le 0.15^2,\\
0.5, &\quad  \mbox{if } (x,y)\in [0.6, 0.8]\times [0.2, 0.5],\\
0, & \quad \mbox{elsewhere}.
\end{array}\right.
$$
We assume that $u(c^\dag)=x+y$, $f=(x+y) c^\dag(x,y)$, and the boundary condition $g=(x+y)|_{\partial \Omega}$.
We add noise to $u(c^\dag)$ to produce a noisy data $u^\d$ satisfying $\|u^\d-u(c^\dag)\|_{L^2(\Omega)}=\d$ with
$\d=0.1\times 10^{-3}$. We take $\a_n=2^{-n}$ and use $u^\d$ to reconstruct $c^\dag$ by our method which is terminated
by Rule \ref{Rule1} with $\tau=1.05$. All partial differential equations involved are solved approximately by a finite
difference method by dividing $\Omega$ into $30\times 30$ small squares of equal size with the resulting linear system solved
by the Gauss--Seidel method. All optimization problems are solved by a restart conjugate gradient method.
We report the computational results in Figure \ref{fig6.3}. In (a) we plot the the exact solution $c^\dag$, in (b)
we plot the computational result corresponding to $\Theta(c)=\|c\|_{L^2}^2$, and in (c) and (d) we plot the computational results
corresponding to $\Theta(c)=\lambda\|c\|_{L^2}^2 +\int_\Omega |D c|$ with $\lambda=0.01$ and $\la=1$ respectively.
We replace $\int_\Omega |Dc|$ by $\int_\Omega \sqrt{|D c|^2 +\varepsilon}$ with $\varepsilon=10^{-6}$ during computation.
It is clear that the reconstruction results in (c) and (d) are much better than the one in (b). Moreover,
the results in (c) and (d) indicate that the method is rather robust with respect to $\lambda$ since the change of
$\lambda$ does not affect the reconstruction much.

\end{example}

\begin{example}\label{e3.2}
Let $\Omega\subset {\mathbb R}^N$ be a bounded domain with Lipschitz boundary $\partial \Omega$. Consider the
identification of the diffusion parameter $a$ in
\begin{equation}\label{4.3.10}
\left\{\begin{array}{lll} -\mbox{div} (a\nabla u)=f \qquad
\mbox{in }
\Omega,\\
u=g \qquad \mbox{on } \partial \Omega
\end{array}\right.
\end{equation}
from the $L^2$ measurement of $u$, where $f\in H^{-1}(\Omega)$ and $g \in H^{1/2}(\partial \Omega)$ are given.
It is well-known that for $a\in L^\infty(\Omega)$ bounded below by a positive constant, (\ref{4.3.10}) has a
unique solution $u=u(a)\in H^1(\Omega)$. We assume that the sought solution $a^\dag$ is in $W^{1,p}(\Omega)$
for some $p>N$ with $a^\dag> \nu_0>0$ on $\Omega$ for some positive constant $\nu_0$. Thus this inverse problem
reduces to solving an equation of the form (\ref{1.1}) if we define $F$ as
\begin{align*}
F:  W^{1,p}(\Omega) \to L^2(\Omega), \qquad  F(a):=u(a)
\end{align*}
with
$$
D(F):=\left\{a\in W^{1,p}(\Omega): a\ge \nu_0 \mbox{ on } \Omega \right\}.
$$
Since $W^{1,p}(\Omega)$ embeds into $L^\infty(\Omega)$, the operator $F$ is well-defined.

This is the inverse groundwater filtration problem corresponding
to the steady state case studied in \cite{H97} in which it has
been shown that $F$ is Fr\'{e}chet differentiable and there holds
\begin{equation}\label{3.3}
\|F(\tilde{a})-F(a)-F'(a) (\tilde{a}-a)\|_{L^2}\lesssim
\|\tilde{a}-a\|_{W^{1,p}} \|F(\tilde{a})-F(a)\|_{L^2}
\end{equation}
for all $\tilde{a}, a\in B_\rho(a^\dag)$, where $B_\rho(a^\dag)$ denotes the ball in $W^{1,p}(\Omega)$
of radius $\rho$ around $a^\dag$.

We will follow the technique in \cite{H97} to show Assumption \ref{A1}(d).
For $\tilde{a}, a\in B_\rho(a^\dag)$ and $h\in W^{1,p}(\Omega)$ we set
\begin{equation}\label{3.3.5}
u=u(a), \quad \tilde{u}=u(\tilde{a}), \quad u'=F'(a) h, \quad
\tilde{u}'=F'(\tilde{a}) h.
\end{equation}
Recall that $u'$ is the weak solution of the boundary value
problem
\begin{align*}
\left\{\begin{array}{lll}  -\mbox{div} (a \nabla u')= \mbox{div} (h \nabla u) \quad
\mbox{in } \Omega,\\
u'=0 \qquad \mbox{on } \partial \Omega.
\end{array} \right.
\end{align*}
The same is true for $\tilde{u}'$. Therefore
\begin{align*}
\left\{\begin{array}{lll}
-\mbox{div}(\tilde{a} \nabla (\tilde{u}'-u'))=\mbox{div} (h
\nabla (\tilde{u}-u))+ \mbox{div} ((\tilde{a}-a)\nabla u') \qquad
\mbox{in } \Omega,\\
\tilde{u}'-u'=0 \qquad \mbox{ on } \partial \Omega.
\end{array}\right.
\end{align*}
Since the operator $A(\tilde{a}): V:=H_0^1\cap H^2(\Omega) \to
L^2(\Omega)$ defined by $A(\tilde{a}) w=-\mbox{div} (\tilde{a}
\nabla w)$ can be extended as an isomorphism $A(\tilde{a}):
L^2(\Omega) \to V'$ so that $A(\tilde{a})^{-1}: V'\to L^2(\Omega)$
is uniformly bounded around $a^\dag$, where $V'$ denotes the
anti-dual of $V$ with respect to the bilinear form (\ref{bilinear}),
from the above equation we then have
\begin{equation}\label{3.4}
\|\tilde{u}'-u'\|_{L^2} \lesssim \|\mbox{div} ((\tilde{a}-a)
\nabla u')\|_{V'} + \|\mbox{div} (h \nabla (\tilde{u}-u))\|_{V'}.
\end{equation}

In order to proceed further, note that for $h\in W^{1,p}(\Omega)$,
$\varphi\in H_0^1(\Omega)$ and $\psi\in V$, we have
\begin{align*}
\int_\Omega \mbox{div}( h \nabla \varphi) \psi dx &= \int_\Omega
\varphi \mbox{div} (h \nabla \psi) dx
\le \|\varphi\|_{L^2} \|\mbox{div}( h \nabla \psi)\|_{L^2}.
\end{align*}
Recall the embedding $W^{1,p}(\Omega) \hookrightarrow L^\infty (\Omega)$ for $p>N$
and the embedding $H^1(\Omega) \hookrightarrow L^q(\Omega)$ for all $q\le 2N/(N-2)$.
Since $p>N$ implies $2p/(p-2)<2N/(N-2)$, we have
\begin{align*}
\|\mbox{div}( h \nabla \psi)\|_{L^2} & \le \|h \Delta \psi\|_{L^2}
+\|\nabla h \cdot \nabla \psi\|_{L^2}\\
&\le \|h\|_{L^\infty} \|\Delta \psi\|_{L^2} + \|\nabla h\|_{L^p}
\|\nabla \psi \|_{L^{\frac{2p}{p-2}}}\\
&\lesssim \|h\|_{W^{1,p}} \|\psi\|_V.
\end{align*}
Therefore, for all $\psi\in V$,
$$
\int_\Omega \mbox{div} (h \nabla \varphi) \psi dx \lesssim
\|h\|_{W^{1,p}} \|\varphi\|_{L^2} \|\psi\|_V
$$
which implies that
$$
\|\mbox{div} (h \nabla \varphi)\|_{V'} \lesssim \|h\|_{W^{1,p}} \|\varphi\|_{L^2}.
$$
Applying this inequality to estimate the two terms on the right
hand side of (\ref{3.4}), we obtain
\begin{equation}\label{3.5}
\|(F'(\tilde{a})-F'(a)) h\|_{L^2} \lesssim \|\tilde{a}-a\|_{W^{1,p}}
\|F'(a) h\|_{L^2} + \|h\|_{W^{1,p}} \|F(\tilde{a})-F(a)\|_{L^2}
\end{equation}
for all $h\in H^2(\Omega)$ and $\tilde{a}, a\in B_\rho(a^\dag)$. From (\ref{3.3}) it follows
$\|F(\tilde{a})-F(a)\|_{L^2}\le 2 \|F'(a)(\tilde{a}-a)\|_{L^2}$ for $\tilde{a}, a\in B_\rho(a^\dag)$
by shrinking the ball $B_\rho(a^\dag)$ if necessary. This together with (\ref{3.5})
verifies Assumption \ref{A1}(d).

In order to reconstruct $a^\dag$, we pick an initial guess $a_0\in W^{1,p}(\Omega)$
and take the function
$$
\Theta(a):=\int_\Omega \left(|a-a_0|^p +|\nabla (a-a_0)|^p \right) dx
$$
which is known to be $\max\{p,2\}$-convex in $W^{1,p}(\Omega)$. Observing that $\xi_0:=0\in \p \Theta(a_0)$
and thus $D_{\xi_0} \Theta(a, a_0)=\Theta(a)$. Therefore, for one-dimensional problem, i.e. $N=1$, we may take
$1<p\le 2$ and define $a_{n+1}^\d$ as the minimizer of the convex functional
\begin{equation}\label{min6.1}
\int_\Omega \left|u^\d-F(a_n^\d) -F'(a_n^\d) (a-a_n^\d)\right|^2 dx
+\a_n \int_\Omega \left(|a-a_0|^p +|\nabla (a-a_0)|^p\right) dx
\end{equation}
over $W^{1,p}(\Omega)$. If $n_\d$ denotes the integer determined by either Rule \ref{Rule1}, \ref{Rule2}, or \ref{Rule3}
with $\tau>1$, we have from Theorem \ref{T5.2} that $\|a_{n_\d}^\d-a^\dag\|_{W^{1,p}}\rightarrow 0$ as $\d\rightarrow 0$. For higher
dimensional problem, i.e. $N\ge 2$, we have $2\le N<p<\infty$. We may define $a_{n+1}^\d$ as the minimizer of the
convex functional
$$
\left(\int_\Omega \left|u^\d-F(a_n^\d) -F'(a_n^\d) (a-a_n^\d)\right|^2 dx\right)^{p/2}
+\a_n \int_\Omega \left(|a-a_0|^p +|\nabla (a-a_0)|^p\right) dx
$$
over $W^{1,p}(\Omega)$. It then follows from Theorem \ref{T4.5.1} that $a_{n_\d}^\d$ converges to $a^\dag$ weakly
in $W^{1,p}(\Omega)$. Since $W^{1,p}(\Omega)$ can be compactly embedded into $L^\infty(\Omega)$, we have
$\|a_{n_\d}^\d-a^\dag\|_{L^\infty(\Omega)}\rightarrow 0$ as $\d\rightarrow 0$.

\begin{figure}[ht!]
     \begin{center}
            \includegraphics[width=1\textwidth, height=2.6in]{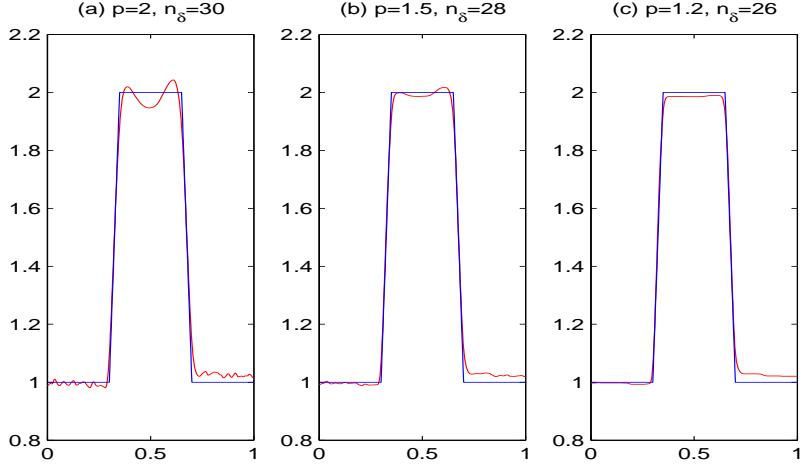}
%
    \end{center}
    \caption{Numerical results for one-dimensional problem in Example \ref{e3.2} with different
    values of $p$ in (\ref{min6.1}), where $n_\d$ denotes the integer determined by Rule \ref{Rule1}
    with $\tau=1.05$.%
             }%
   \label{f6.1}
\end{figure}

In the following we present a numerical test for the one-dimensional problem over the interval $\Omega=[0,1]$ with
boundary data $u(0)=u(1)=0$ and inhomogeneous term
$$
f(t)=\left\{\begin{array}{lll}
-2, & \quad   0\le t\le 0.3,\\
30-80t, & \quad  0.3< t< 0.35,\\
-4, & \quad 0.35\le t\le 0.65,\\
80t-50, & \quad  0.65<t<0.7,\\
-2, & \quad  0.7\le t\le 1.
\end{array}\right.
$$
The function to be reconstructed is
$$
a^\dag(t)=\left\{\begin{array}{lll}
1, & \quad   0\le t\le 0.3,\\
20t-5, & \quad  0.3< t< 0.35,\\
2, & \quad 0.35\le t\le 0.65,\\
15-20t, & \quad  0.65<t<0.7,\\
1, & \quad  0.7\le t\le 1.
\end{array}\right.
$$
Observing that $u(a^\dag)=t(t-1)$. We add noise to $u(a^\dag)$ to produce a noisy data $u^\d$ satisfying
$\|u^\d-u(a^\dag)\|_{L^2[0,1]}=\d$ with given noise level $\d>0$ and use $u^\d$ to reconstruct $a^\dag$
by our method in which each iterate is defined by the convex optimization problem (\ref{min6.1})
with $1<p\le 2$. We take the noise level $\d=0.1\times 10^{-3}$ and the initial guess $a_0=1$. We also take
the sequence $\{\a_n\}$ to be $\a_n=2^{-n}$. During the computation, all differential equations are solved
approximately by the finite element method on the subspace of piecewise linear splines on a uniform grid
with subinterval length $1/400$, and the optimization problems (\ref{min6.1}) are solved by a restart
conjugate gradient method (\cite{SY2006}). In Figure \ref{f6.1} we report the numerical results of our method for several
different values of $p\in (1,2]$ with the iteration terminated by Rule \ref{Rule1} with $\tau=1.05$. It shows
that the method works well for these selected values of $p$. Moreover, by decreasing $p$ from $2$ to $1.2$,
the reconstruction result becomes better when the sought solution has corners and
constant parts. However, one has to pay the price of more computational time for smaller $p$.

\end{example}

\begin{example}\label{e3.4}
We consider the transient case of the inverse groundwater
filtration problem which identifies the material coefficient $a$
in
\begin{align}\label{3.10}
\left\{ \begin{array}{lll}
\frac{\partial u}{\partial t}- \mbox{div} (a \nabla u) =f \quad
\mbox{ in } \Omega \times (0, T], \\
u=\varphi \qquad \mbox{ on } \partial \Omega \times (0, T], \\
u=u_0 \qquad \mbox{ on } \Omega \times \{t=0\}
\end{array} \right.
\end{align}
from the $L^2(0, T; L^2(\Omega))$-measurement of $u$, where $\Omega \subset {\mathbb R}^N$ is a bounded
domain with Lipschitz boundary,  $f\in L^2(0, T; H^{-1}(\Omega))$, $\varphi\in L^2(0, T; H^{1/2}(\partial \Omega))$
and $u_0\in H^1(\Omega)$. It is well-known that (\ref{3.10}) has a unique solution $u:=u(a)\in L^2(0, T; H^1(\Omega))$
for each $a\in L^\infty(\Omega)$ bounded from below by a positive constant.
We assume that the sought solution $a^\dag$ is in $W^{1,p}(\Omega)$ with $p>N$ satisfying
$a^\dag>\nu_0>0$ on $\Omega$. This inverse problem reduces to solving (\ref{1.1}) if we define the nonlinear operator
$F: W^{1,p}(\Omega) \to L^2(0, T; L^2(\Omega))$ by $F(a):=u(a)$ with the same domain $D(F)$ as in Example \ref{e3.2}.
It is known that $F$ is Fr\'{e}chet differentiable, and, for $a\in D(F)$ and $h\in W^{1,p}(\Omega)$, $u':=F'(a) h$ satisfies
\begin{align*}
\left\{ \begin{array}{lll}
\frac{\partial u'}{\partial t} -\mbox{div} (a \nabla
u')=\mbox{div} (h \nabla u) \qquad \mbox{in } \Omega\times (0, T],\\
u'=0 \qquad \mbox{on } \partial \Omega \times (0, T],\\
u'=0 \qquad \mbox{on } \Omega \times \{t=0\}.
\end{array} \right.
\end{align*}
Using the same notations as in (\ref{3.3.5}) we have for
$w:=\tilde{u}'-u'$ that
\begin{align*}
\left\{\begin{array}{lll}
\frac{\partial w}{\partial t}- \mbox{div}(\tilde{a} \nabla w)
=\mbox{div} ((\tilde{a}-a) \nabla u') +\mbox{div}( h\nabla
(\tilde{u}-u)) \quad \mbox{in } \Omega\times (0, T],\\
w=0 \qquad \mbox{on } (\partial \Omega \times (0, T]) \times
(\Omega\times \{t=0\}).
\end{array} \right.
\end{align*}
From the well-known facts on parabolic equations (see \cite{W}) it
follows that
$$
\|w\|_{L^2(0, T; L^2(\Omega))} \lesssim \|\mbox{div}
((\tilde{a}-a) \nabla u')\|_{L^2(0, T; V')} + \|\mbox{div} (h
\nabla (\tilde{u}-u) )\|_{L^2(0, T; V')}
$$
for all $\tilde{a}$ and $a$ in a neighborhood around $a^\dag$. By
employing the corresponding estimates derived in Example
\ref{e3.2} we obtain
\begin{align}\label{3.11}
\|w\|_{L^2(0, T; L^2(\Omega))} & \lesssim
\|\tilde{a}-a\|_{W^{1,p}(\Omega)} \|u'\|_{L^2(0, T; L^2(\Omega))} \nonumber\\
& \quad \, + \|h\|_{W^{1,p}(\Omega)} \|\tilde{u}- u\|_{L^2(0, T; L^2(\Omega))}.
\end{align}
From \cite[Theorem 3.2]{H97} we know that
$$
\|\tilde{u}-u\|_{L^2(0, T; L^2(\Omega))} \lesssim \|F'(a)
(\tilde{a} -a)\|_{L^2(0, T; L^2(\Omega))}.
$$
This together with (\ref{3.11}) implies Assumption \ref{A1}(c). Therefore, our method is applicable to
this example, and we can formulate the procedure to reconstruct $a^\dag$ similarly as is done in Example \ref{e3.2}.
\end{example}

\vskip 0.3cm
\noindent
{\bf Acknowledgements}  Q Jin is partly supported by the grant DE120101707 of Australian Research Council,
and M Zhong is partly supported by the National Natural Science Foundation of China (No.11101093).

\end{document}